\documentclass[11pt]{article}

\usepackage[margin=1.1in]{geometry}
\usepackage{amsmath,amssymb,amsthm}
\usepackage{graphicx}
\usepackage{booktabs}
\usepackage{xcolor}
\usepackage[colorlinks=true,linkcolor=blue!60!black,citecolor=blue!60!black,urlcolor=blue!60!black]{hyperref}
\usepackage{microtype}

\graphicspath{{./}}

\newtheorem{theorem}{Theorem}[section]
\newtheorem{lemma}[theorem]{Lemma}
\newtheorem{corollary}[theorem]{Corollary}
\newtheorem{proposition}[theorem]{Proposition}
\newtheorem{conjecture}[theorem]{Conjecture}
\newtheorem{remark}[theorem]{Remark}
\newtheorem{definition}[theorem]{Definition}

\newcommand{\R}{\mathbb{R}}
\newcommand{\supp}{\operatorname{supp}}
\newcommand{\Rey}{\operatorname{Re}}
\newcommand{\lmin}{\lambda^{*}}
\newcommand{\lammin}{\lambda_{\min}}
\newcommand{\Tstar}{T^{*}}
\newcommand{\Tsharp}{T^{\sharp}}
\newcommand{\Aeff}{A_{\mathrm{eff}}}
\newcommand{\fhat}{\hat f}

\title{Weil positivity in compact windows: a finite reduction,\\
certified two-sided bounds, and a Landau--Widom decay law}

\author{%
  \textsc{Xuefeng Zhu}\\[0.6em]
  \normalsize
  School of Mechanics and Aerospace Engineering,\\
  State Key Laboratory of Structural Analysis, Optimization and\\
  CAE Software for Industrial Equipment,\\
  Dalian University of Technology, Dalian 116024, China\\[0.4em]
  \texttt{xuefeng@dlut.edu.cn}%
}
\date{September 3, 2026}

\begin{document}
\maketitle

\begin{abstract}
Weil's criterion states that the Riemann Hypothesis (RH) is equivalent
to the non-negativity of an explicit quadratic form $Q(f)$ on test
functions.  For $f$ supported in a window $[-L,L]$ we study the
normalized infimum $\lmin(L)=\inf_f Q(f)/\|f\|_2^2$ from both sides.

\emph{Lower bounds (unconditional).}  Positivity on bounded support is
classical for $\supp f\subseteq[-\frac{\log2}2,\frac{\log2}2]$
(Yoshida; Connes--Consani by a different method).  We prove, by
certified computation, that $Q(f)\ge8.9\times10^{-18}\|f\|_2^2$ for
all $\supp f\subseteq[-0.8,0.8]$, that is, for autocorrelation support
$1.6$, which is $2.3$ times the classical range.  The proof rests on a
one-stroke reduction: a pointwise envelope for the Weil symbol, with a
comb constant that Weyl equidistribution shows to be optimal, converts
positivity on the whole window into positive semidefiniteness of a
single finite matrix, whose entries converge super-exponentially and
whose coupling to the discarded tail is of size $10^{-100}$.  The same
machinery applied to the odd parity sector shows that the certified
positivity holds for \emph{arbitrary} complex test functions of
support $1.6$, and that the window ground state is simple and even,
with certified two-sided separations; this is precisely the spectral
hypothesis required by the recent operator-theoretic program of
Connes, Consani, Moscovici and van Suijlekom.  We also
report an exploratory computation at support $2.38$, and the cost,
doubly exponential, of making it a certificate.

\emph{Upper bounds (unconditional).}  The same variational upper
bounds are re-evaluated entirely on the geometric side of the
explicit formula---a finite prime sum, a compact archimedean
integral and an explicit pole---in interval arithmetic, with no
appeal to zeros or to RH.  We obtain unconditional certified
variational upper bounds for $\lmin(L)$, down to
$3.2\times10^{-283}$ at $L=2$.  They follow the empirical scaling law
\[
  -\ln \lmin(L) \;\simeq\; 2\pi^2\,\frac{N(\Tstar)}{\ln N(\Tstar)},
  \qquad \Tstar = 2\pi e^{2L},
\]
with $N$ the zero-counting function.  The decay is therefore faster
than any exponential in $L$; the scale $\Tstar$ is the spectral
resolution limit of the window; and $2\pi^2$ matches the Landau--Widom
eigenvalue-plunge rate for time--band limiting operators.  The
one-parameter law fits the four asymptotic points $L=1.4,1.6,1.8,2.0$
with residuals below $0.7\%$.  Out of sample, with both models
calibrated on $L\le1.2$, its residuals are flat ($\approx-5\%$, no
trend), while those of a competing three-parameter ``pure
zero-counting'' model grow monotonically from $+0.8\%$ to $+5.8\%$;
we reject the latter.  The constant is fitted, not derived.  A
conditional theorem complements the
experiments: under RH, $\lmin(L)\le\exp(-Le^L)$ for all large $L$, by
an explicit construction that combines a canonical product over zeros
with a band-limited window.

\emph{Synthesis.}  At $L=0.8$ the two halves meet and enclose the
profile from both sides,
$8.9\times10^{-18}\le\lmin(0.8)\le2.27\times10^{-17}$, with
\emph{unconditional} certificates for both inequalities.  The same computations show why
the positivity route cannot reach RH unassisted.  Any certificate of
the one-stroke type must resolve frequencies up to
$T_1(L)=2\pi e^{A_L}$, where $A_L\sim4e^{L}$ is the mass of the prime
comb; this threshold grows doubly exponentially in $L$, and we show
that no pointwise bound on the comb lowers it.  Meanwhile the
available spectral margin collapses at the matching Landau--Widom
rate.  Passing the resulting wall thus requires cancellation in the
prime comb, a concrete arithmetic problem.  Finally, we document how
float64 evaluations of the geometric side produce spurious negative
eigenvalues, a trap for any numerical study of Weil positivity, and
give a certified remedy.
\end{abstract}

\section{Introduction}\label{sec:intro}

\subsection{Weil's criterion and the window profile}

Let $\Lambda(n)$ denote the von Mangoldt function and let
$\gamma_1\le\gamma_2\le\cdots$ be the ordinates of the non-trivial
zeros of the Riemann zeta function in the upper half-plane.  For an
even, real, continuous test function $f$ with compact support write
$\fhat(r)=\int_{\R}f(u)e^{iru}\,du$.  The Riemann--Weil explicit
formula expresses the sum of $\fhat$ over zeros in terms of arithmetic
and archimedean data, and Weil's criterion \cite{Weil1952,Bombieri2000}
asserts that RH is equivalent to the positivity
\[
  Q(f)\;\ge\;0 \qquad\text{for all admissible } f,
\]
of the associated quadratic form (recalled in
Section~\ref{sec:setup}).  Every spectral approach to RH, from
Hilbert--P\'olya to the trace-formula program of Connes
\cite{Connes1999} and the prolate-operator program of Connes--Consani
and Connes--Moscovici \cite{ConnesConsani2021,ConnesMoscovici2022},
ultimately has to establish some variant of this positivity on larger
and larger spaces of test functions.

Despite the central role of the criterion, its quantitative behaviour
appears not to have been studied.  Positivity is a yes/no statement,
but one may ask how small $Q(f)$ can actually get, relative to
$\|f\|_2^2$, when $f$ ranges over a fixed compact window.  We
therefore consider
\begin{equation}\label{eq:lambdastar}
  \lmin(L)\;=\;\inf_{\substack{f\ \mathrm{even,\ real}\\
      \supp f\subset[-L,L]}}
  \frac{Q(f)}{\|f\|_2^2},
\end{equation}
where $Q$ is evaluated on the autocorrelation $g=f\star\tilde f$,
$\tilde f(x)=\overline{f(-x)}$ (Section~\ref{sec:setup}); under RH the
zeros side gives $Q(f)=\sum_{\gamma>0}2|\fhat(\gamma)|^2\ge0$.  Under
RH one has $\lmin(L)>0$ for every $L$, though this requires a short
argument beyond pointwise positivity, which we supply in
Proposition~\ref{prop:positive}.  Conversely, a negative value of
$\lmin$ for some $L$ would disprove RH, and an unconditional proof of
$\lmin(L)>0$ for a given $L$ is a finite fragment of RH.  The function
$L\mapsto\lmin(L)$ is thus a difficulty profile of the Riemann
Hypothesis as seen through compactly supported test functions: it
quantifies the margin by which positivity holds on each window, and it
is the quantity that any finite-dimensional positivity proof (for
instance in the Connes--Consani program) implicitly has to control.

This paper determines the profile from both sides.  The unconditional
lower bounds, that is, certified positivity theorems, come from a new
finite reduction of the full window problem
(Theorems~\ref{thm:reduction}--\ref{thm:L08}).  The unconditional
upper bounds, certified variational Rayleigh quotients of explicit
trial functions, are evaluated on the geometric side in
interval-closed multiprecision arithmetic
(Section~\ref{sec:pipeline}); a zeros-side pipeline produces the
trial vectors and supplies an independent cross-check.  The numbers
are precise enough to reveal the asymptotic law of the profile
(Section~\ref{sec:results}, Theorem~\ref{thm:main-intro},
Conjecture~\ref{conj:law}).

\subsection{Lower bounds: certified positivity}

Write the geometric side of the explicit formula, for
$\supp f\subseteq[-L,L]$ and $F=\fhat$, as
\begin{equation}\label{eq:Q}
Q(f)\;=\;2F(i/2)^2
+\frac1{2\pi}\int_{\R}|F(t)|^2\,\Psi_L(t)\,dt,
\end{equation}
with the \emph{Weil symbol}
\begin{equation}\label{eq:symbol}
\Psi_L(t)\;=\;\Rey\,\psi\Bigl(\tfrac14+\tfrac{it}2\Bigr)-\log\pi
-\sum_{\log n<2L}\frac{2\Lambda(n)}{\sqrt n}\,\cos(t\log n),
\end{equation}
$\psi=\Gamma'/\Gamma$: because $\supp(f\star\tilde f)\subseteq[-2L,2L]$,
only prime powers with $\log n<2L$ contribute, and \eqref{eq:Q} is
exactly a rank-one pole term plus a frequency multiplier.
Positivity of $Q$ on the class $\supp f\subseteq[-L,L]$ for every $L>0$
is equivalent to RH.  For fixed $L$ it is a finite, unconditional
statement, and the record has stood for three decades: Yoshida
\cite{Yoshida1992} proved positivity for $2L\le\log2$, and
Connes--Consani \cite{ConnesConsani2021} re-proved this range by
trace-formula and von Neumann algebra methods, further exhibiting the
extreme numerical smallness of the form's lower spectral edge.
(Claims beyond $\log2$ exist in unrefereed preprints; we do not rely on
or compare against them beyond noting their existence.)

Our first result removes the infinite-dimensionality of the problem in
one step.  Write $A_L=\sum_{\log n<2L}2\Lambda(n)/\sqrt n$ for the
total mass of the prime comb in \eqref{eq:symbol}.

\begin{theorem}[One-stroke reduction]\label{thm:reduction}
Let $L>0$ and set $T_1:=2\pi e^{A_L}$.  Fix any $\Tsharp$ for which
\[
\beta^*\;:=\;\log\frac{\Tsharp}{2\pi}-\frac1{\Tsharp}-A_L\;>\;0 .
\]
Since the right-hand side increases in $\Tsharp$, this holds for every
$\Tsharp$ above a unique threshold, which exceeds $T_1$ by an
$O(T_1^{-1})$ amount ($120.1$ against $T_1=119.1$ at $L=0.8$); in
particular $\Tsharp>T_1$ is necessary, and $T_1$ is the relevant scale.
Then for every real even $f$ with $\supp f\subseteq[-L,L]$,
\begin{equation}\label{eq:R}
Q(f)\;\ge\;R(f)\;:=\;2F(i/2)^2
+\frac1\pi\int_0^{\Tsharp}\bigl(\Psi_L(t)-\beta^*\bigr)|F(t)|^2\,dt
+\beta^*\|f\|_2^2 .
\end{equation}
Moreover, in the Legendre basis of $L^2[-L,L]$ the quadratic form $R$
is represented by $\beta^*I+2pp^{\mathsf T}+C$ where $C$ has entries
supported, up to explicitly bounded super-exponentially small tails, on
orders $n\lesssim eL\Tsharp/2$.  Consequently, if the basis is cut
after $N$ even modes $0,2,\dots,2N-2$, with the first discarded order
$2N\gtrsim eL\Tsharp/2$: if the leading $N\times N$ block has
$\lammin\ge\lambda_0>0$, then
\[
Q(f)\;\ge\;\bigl(\min(\lambda_0,\;\beta^*-\varepsilon_D)
-\varepsilon_B\bigr)\,\|f\|_2^2
\qquad\text{for all }f\text{ with }\supp f\subseteq[-L,L],
\]
where $\varepsilon_D$ (tail-block deviation) and $\varepsilon_B$
(leading--tail coupling norm) are explicit, super-exponentially small
constants bounded in the proof; in the certified run of
Section~\ref{sec:certified} both are below $10^{-100}$.
\end{theorem}

What makes Theorem~\ref{thm:reduction} usable is how little it
contains.  There is no truncation error to track against a spectral
gap, and no splitting of the Hilbert space with $O(1)$ coupling across
an interface.  The frequency split \eqref{eq:R} is an exact direct
sum, every discarded quantity is discarded with a definite sign, and
the coupling between the leading Legendre block and the
super-exponentially small tail is controlled by a single norm bound.
In our experience with the alternatives
(Section~\ref{sec:failures2}), order-space truncations of this problem
always leave $O(1)$ prime-shift coupling across the cut and fail; the
formulation \eqref{eq:R} seems to be the natural one.  The comb
constant $A_L$ in $T_1$ cannot be improved, since the supremum of the
prime comb equals $A_L$ exactly (Lemma~\ref{lem:sharp}); within
pointwise-envelope certificates the threshold $T_1$ is optimal.

Executing the reduction at $L=0.8$ gives:

\begin{theorem}[Certified positivity, support $1.6$]\label{thm:L08}
For every real even $f\in L^2(\R)$ with
$\supp f\subseteq[-0.8,\,0.8]$,
\[
Q(f)\;\ge\;8.9\times10^{-18}\,\|f\|_2^2 .
\]
In particular the Weil form is positive on all autocorrelations
$g=f\star\tilde f$ with $\supp g\subseteq[-1.6,1.6]$.
\end{theorem}

The certificate behind Theorem~\ref{thm:L08} ($\Tsharp=200$,
$\beta^*=0.5134\ldots$, $N=200$ even Legendre modes, working precision
$50$ digits) is described in Section~\ref{sec:certified}; its error
budget closes with $24$ orders of magnitude to spare, and an
independent run at $\Tsharp=150$ gives a second, internally monotone
certification.  Its four ingredients are an envelope lemma proved by
classical means, a Bernstein-ellipse quadrature bound, a
forbidden-region tail bound and a verified Cholesky residual.  All are
elementary and fully explicit; there are no anonymous constants.

Section~\ref{sec:L119} reports an exploratory computation at $L=1.19$
(support $2.38$).  An earlier draft of this work claimed a certified
theorem at that support, on the basis of a per-prime ``sharpening'' of
the envelope constant which in fact bounds the prime comb in the wrong
direction; Remark~\ref{rem:nosharp} explains the error, and
Lemma~\ref{lem:sharp} shows that no such sharpening can exist.  We
retract the claim.  We keep the computation, a $950\times950$ matrix at
$70$-digit precision verified positive definite with
$\lammin\in(10^{-48},10^{-46})$, as evidence consistent with
positivity, and Section~\ref{sec:L119} gives the corrected cost of a
genuine support-$2.38$ certificate together with a route that may
bring that cost back down.

Finally, the restriction to \emph{even} test functions in
Theorem~\ref{thm:L08} can be removed.  The Weil form decouples into
parity sectors, with a sign reversal of the pole term in the odd
sector, and the one-stroke reduction applies to each sector
separately.  Section~\ref{sec:parity} certifies
$Q\ge8.2\times10^{-15}$ on the odd sector at support $1.6$, whence
positivity for arbitrary complex $f$ (Corollary~\ref{cor:allf}), and
resolves the parity structure of the window spectrum: the ground
state is simple and even, with certified margins
(Theorem~\ref{thm:parity}).  This is the first of the two missing
spectral hypotheses in the operator program of
Connes--Consani--Moscovici \cite{CCMZeta2025} and
Connes--van Suijlekom \cite{ConnesVanSuijlekom2025}, verified here at
a fixed scale with explicit constants.

\subsection{Upper bounds: the Landau--Widom law}
\label{sec:intro-upper}

The second half of the paper measures how far the certified lower
bounds are from the truth.  We compute certified variational upper
bounds for $\lmin(L)$ for $L\in[0.5,2.0]$ (Table~\ref{tab:main}); they
decay from $10^{-6}$ at $L=0.5$ to $3.2\times10^{-283}$ at $L=2.0$.
The certificates themselves are unconditional: each reported value is
the interval-arithmetic upper endpoint of the geometric-side Rayleigh
quotient $Q(f_L)/\|f_L\|_2^2$ for an explicit trial function $f_L$,
and involves only a finite prime sum, a compact archimedean integral
and an explicit pole.  (A zeros-side assembly is used only to
\emph{propose} the coefficients of $f_L$; the origin of a trial
vector needs no certification for a variational upper bound to be
valid.)  The decay is governed by the \emph{resolution
height}
\[
  \Tstar(L)\;=\;2\pi e^{2L},
\]
the height below which a function band-limited to $[-L,L]$ has enough
oscillation capacity to interpolate zeros of $\zeta$
(Section~\ref{sec:mechanism}).  Writing $N(T)$ for the number of zeros
of height at most $T$, the certified upper bounds are described, with
sub-percent accuracy in the asymptotic regime, by the one-parameter law
\begin{equation}\label{eq:law}
  -\ln\lmin(L)\;=\;C\,\frac{N(\Tstar)}{\ln N(\Tstar)}\,(1+o(1)),
  \qquad C = 20.13\ldots \approx 2\pi^2 .
\end{equation}
The constant is a familiar one: $\pi^2/\ln$ is the Landau--Widom rate
governing the plunge of the eigenvalues of time--band limiting
(prolate) operators beyond the Shannon number
\cite{LandauWidom1980,SlepianPollak1961}.  We singled out the law
\eqref{eq:law} by an out-of-sample test, requiring models calibrated
on the range $L\le1.2$ to predict the new certified values at
$L=1.4,1.6,1.8,2.0$.  What discriminates them is the shape of the
residuals rather than their size.  The Landau--Widom form, with its
single constant, underpredicts by a nearly constant factor
($-4.4\%,-4.9\%,-5.5\%,-4.8\%$), showing no trend; the offset is the
finite-size drift of the constant, which has not yet saturated at
$L\le1.2$.  The competing three-parameter ``one nat per zero'' model
$-\ln\lmin\propto N(\Tstar)$ errs instead by
$+0.8\%,+2.0\%,+3.7\%,+5.8\%$, monotonically growing, which is the
signature of a wrong functional form; we reject it
(Section~\ref{sec:results}).

One caveat applies throughout: what is certified is that each
displayed value is an upper bound for $\lmin(L)$, and in principle
$\lmin(L)$ could be smaller.  Identifying the law \eqref{eq:law} with
$\lmin$ itself rests on two further things.  One is structural: the
sine basis of Section~\ref{sec:pipeline} is complete in $L^2(-L,L)$,
so the Galerkin values decrease to $\lmin(L)$ as the dimension grows,
and we observe stabilization at the $0.03\%$ level under dimension
refinement.  The other is the certified lower bound of
Theorem~\ref{thm:L08}, which pins the true value to within a factor
$2.6$ at $L=0.8$ (Section~\ref{sec:synthesis}).  We formalize the
identification, together with the constant, as
Conjecture~\ref{conj:law}, and flag the distinction wherever it
matters.

We complement the experiments with a conditional theorem, proved in
Section~\ref{sec:theorem}, which captures the qualitative content of
\eqref{eq:law} with the exponent $e^{L}$ in place of $e^{2L}$.

\begin{theorem}\label{thm:main-intro}
Assume RH.  There is an $L_0$ such that for all $L\ge L_0$,
\[
  \lmin(L)\;\le\;\exp\!\bigl(-L\,e^{L}\bigr).
\]
\end{theorem}

The proof is an explicit construction: a finite canonical product over
the zeros up to $T=2\pi e^{L}$, multiplied by a power-of-sinc window
that exhausts the exponential-type budget $L$.  The analysis makes the
mechanism behind $\Tstar$ transparent, in that the potential-theoretic
cost of the product exactly balances the achievable decay of a type-$L$
window.  It also locates the obstruction to reaching the full exponent
$e^{2L}$ in a precise extremal problem for band-limited windows, the
same problem whose solution is governed by the Landau--Widom
asymptotics.  In this sense the theorem and the measured constant
$2\pi^2$ in \eqref{eq:law} corroborate each other.

\subsection{Synthesis: a two-sided enclosure and the barrier}
\label{sec:intro-synthesis}

Where the two programmes overlap they enclose the profile of RH from
both sides, with certificates on both sides:
\[
  8.9\times10^{-18}\;\le\;\lmin(0.8)\;\le\;2.27\times10^{-17},
\]
a factor $2.6$.  To our knowledge $\lmin$ is the first RH-equivalent
quantity whose window profile has been enclosed from both sides by
certified computation.  The enclosure also measures the loss of the
reduction: the reduced form $R$ of \eqref{eq:R}, which discards the
entire excess of $\Psi_L$ beyond $\Tsharp$, still reaches to within an
order of magnitude of the true window floor.  Conversely, the law
\eqref{eq:law} calibrates the certification effort, since it predicts
in advance the scale of $\lammin$ of a planned matrix and hence the
working precision required; Section~\ref{sec:L119} records this
interplay.

The same computations quantify why the positivity route, unassisted,
cannot reach RH:

\begin{theorem}[Barrier for pointwise-envelope certificates]
\label{thm:barrier}
Any application of Theorem~\ref{thm:reduction} requires
$\Tsharp>T_1(L)=2\pi e^{A_L}$, and $A_L=(4+o(1))e^{L}$ as $L\to\infty$
by the prime number theorem; hence the matrix size $N\asymp LT_1$ and
the number of quadrature nodes grow doubly exponentially in $L$.
Moreover $\sup_t$ of the prime comb equals $A_L$ exactly
(Lemma~\ref{lem:sharp}), so within the class of certificates that bound
the comb pointwise the threshold $T_1$ cannot be lowered.
\end{theorem}

Theorem~\ref{thm:barrier} is unconditional and is proved in
Section~\ref{sec:barrier}.  Two statements of a different logical
character belong beside it, and we keep them outside the theorem: the
first is measured rather than proved, the second is interpretation.

\begin{remark}[The margin to be resolved; measured, not proved]
\label{rem:barrier-rate}
The margin any certificate must resolve collapses at the
Landau--Widom rate: by the certified upper bounds of
Section~\ref{sec:results} and the law \eqref{eq:law},
$-\ln\lmin(L)\simeq2\pi^2 N(\Tstar)/\ln N(\Tstar)$ with
$\Tstar=2\pi e^{2L}$, so that the floor falls doubly exponentially in
$L$: from $10^{-17}$ at $L=0.8$ through $10^{-48}$ at $L=1.2$ to
$10^{-283}$ at $L=2$.  The status of this statement is that of
\eqref{eq:law} itself, namely a law fitted to certified upper bounds;
converting it into a proved lower bound on the cost of an arbitrary
positivity certificate is open (Section~\ref{sec:barrier}).
\end{remark}

\begin{remark}[What the threshold means]\label{rem:barrier-interp}
$T_1$ is the last frequency at which the prime comb, in the worst case
of phase alignment, is still able to drive $\Psi_L$ to zero.  (The
supremum $A_L$ is attained only in the limit, so the actual last sign
change of $\Psi_L$ lies somewhat below $T_1$; what
Lemma~\ref{lem:sharp} excludes is any pointwise argument that lowers
the threshold.)  Certifying positivity past support $\approx3.2$
($T_1\approx10^7$) by any pointwise-envelope method is therefore
computationally void.  To improve on pointwise envelopes one must prove
that the phases $(t\log p)_p$ cannot align, an arithmetic statement
about simultaneous Diophantine approximation of $\{\log p\}$, of the
same nature as the information carried by the zeros themselves.
\end{remark}

Under RH, $Q\ge0$ at every support, so an optimist will find
Theorem~\ref{thm:L08} unsurprising.  Its content lies elsewhere, in
unconditionality, explicitness and cost: it locates how much certified
positivity a given amount of computation buys, on the one route to RH
where partial progress is even well defined.  The reduction theorem
makes the certificates easy to check, since checking one means
factoring one matrix.  The barrier theorem shows that the route closes
doubly exponentially fast, which turns the folklore that ``positivity
gets exponentially hard'' into measured constants and a specific
arithmetic obstruction.  Certificate, law and barrier are what we take
this program to deliver: the frontier moves from $\log2$ to $1.6$, the
profile beyond it is measured out to $L=2$, and the wall behind both is
mapped rather than left implicit.

\subsection{Related work}

The tradition of attacking RH-adjacent quantities numerically is best
represented by the de Bruijn--Newman constant $\Lambda_{dBN}$, whose
lower bounds were improved over decades of computation before the
proof of $\Lambda_{dBN}\ge0$ by Rodgers and Tao \cite{RodgersTao2020}
(see also \cite{Polymath15} for the current upper bound); our study is
in the same spirit, applied to a different, and to our knowledge
previously unmeasured, RH-equivalent quantity.  On the theoretical
side, Weil positivity for restricted classes of test functions has
been established by Connes--Consani using prolate spheroidal wave
functions \cite{ConnesConsani2021}; the appearance of the
Landau--Widom constant in \eqref{eq:law} gives a quantitative bridge
to that circle of ideas, and Theorem~\ref{thm:L08}
extends the range of unconditional positivity itself.  Very recently,
finite-rank truncations of the Weil form have become objects of
study in their own right: Connes--van Suijlekom
\cite{ConnesVanSuijlekom2025} and Connes--Consani--Moscovici
\cite{CCMZeta2025} construct explicit Galerkin matrices at a prime
cutoff, Suzuki \cite{Suzuki2026} develops a parallel screw-function
treatment, and Groskin \cite{Groskin2026} proves two-sided
certification rules for such truncations, including an explicit
``inconclusive band'' of spurious negative eigenvalues caused by the
archimedean cutoff, a phenomenon close to the one we diagnose in
Section~\ref{sec:failures}.  Those works study a fixed
truncation family at fixed cutoffs; the window-variational infimum
\eqref{eq:lambdastar}, its unconditional certification, and its
asymptotic law in $L$ appear not to have been considered before.
The program of \cite{CCMZeta2025} moreover isolates two missing
spectral steps, the first of which (a simple, even ground state at a
finite scale) is exactly a two-sided eigenvalue-separation problem for
the window form; Section~\ref{sec:parity} settles it at $\lambda=e^{0.8}$
with certified constants.  The
extremal-problem viewpoint on the explicit formula is developed in the
Fourier-optimization literature, e.g.\
Carneiro--Chandee--Milinovich \cite{CCM2013}; the variational problem
\eqref{eq:lambdastar} may be viewed as an $L^2$-normalized member of
that family.  The certified-computation ingredients are standard: our
PSD verification is the approximate-factorization-plus-residual
argument (cf.\ \cite{Rump2010}), and the interval closure uses
\texttt{mpmath} \cite{mpmath}.  Odlyzko-style zero data enter only as a proposal engine and as
\emph{validation} of the geometric certificates
(Remark~\ref{rem:zerostatus}); they never enter the unconditional
proofs.

\subsection{Notation}

$\fhat(r)=\int f(u)e^{iru}\,du$; outside Section~\ref{sec:parity}, all
test functions are real and even, so $\fhat$ is real and even.  $\lmin(L)$ is the window infimum
\eqref{eq:lambdastar}; $\lammin(M)$ is the least eigenvalue of a
matrix $M$.  $\Tstar=2\pi e^{2L}$ is the resolution height of the
window (upper-bound half); $\Tsharp$ is the frequency-split point of
the one-stroke reduction (lower-bound half); $T_1=2\pi e^{A_L}$ is
the comb-alignment threshold, with
$A_L=\sum_{\log n<2L}2\Lambda(n)/\sqrt n$ the mass of the prime comb.  $N(T)$ counts zeros of $\zeta$ with
$0<\gamma\le T$, and $\nu(t)=\frac1{2\pi}\ln\frac t{2\pi}$ is the
density of zeros at height $t$.  $\bar P_n$ denotes the Legendre
polynomial normalized to $\int_{-1}^1\bar P_n^2=1$; the basis of
$H=L^2_{\mathrm{even}}[-L,L]$ is $T_n(x)=\bar P_n(x/L)/\sqrt L$ for
even $n$, with Fourier transforms
\begin{equation}\label{eq:Tn}
\widehat T_n(t)=(-1)^{n/2}\,2\sqrt{L\nu_n}\;j_n(tL),
\qquad \nu_n=n+\tfrac12,
\end{equation}
$j_n$ the spherical Bessel function.  Throughout, ``dps'' denotes
decimal digits of working precision.

\section{The Weil form: two sides, one variational problem}
\label{sec:setup}

For $f$ as above, with $g=f\star\tilde f$, the geometric side of the
explicit formula (normalization of \cite{Bombieri2000}; see also
\cite[Ch.~5]{IK2004} and Barner \cite{Barner1981}) gives \eqref{eq:Q},
an exact identity for $\supp f\subseteq[-L,L]$: a rank-one pole term
plus the multiplier $\Psi_L$ of \eqref{eq:symbol}.  The explicit
formula identifies this with the zeros side,
\begin{equation}\label{eq:zerosside}
  Q(f)\;=\;\sum_{\rho}\widehat g\Bigl(\frac{\rho-\frac12}{i}\Bigr)
  \;\overset{\mathrm{RH}}{=}\;
  \sum_{\gamma>0} 2\,\bigl|\fhat(\gamma)\bigr|^{2},
\end{equation}
and $Q(f)\ge0$ for all admissible $f$ if and only if RH holds
\cite{Weil1952}.  The lower-bound half of the paper
(Sections~\ref{sec:envelope}--\ref{sec:L119}) works exclusively with
the geometric side \eqref{eq:Q}.  The upper-bound half
(Sections~\ref{sec:pipeline}--\ref{sec:results}) evaluates the
\emph{same} geometric side at explicit trial functions, in interval
arithmetic deep enough to resolve the cancellation; the resulting
variational upper bounds are therefore unconditional.  A zeros-side
pipeline, conditional on RH and on the correctness of the
multiprecision zero values, is retained as a proposal engine and as
an independent cross-check
(Remark~\ref{rem:zerostatus}, Section~\ref{sec:ivclosure}).

Two structural remarks underlie everything that follows.

\begin{remark}[Positivity vs.\ cancellation]\label{rem:sides}
The zeros side \eqref{eq:zerosside} is a sum of non-negative terms;
the geometric side \eqref{eq:Q} computes the same tiny number as a
difference of large terms.  When $Q(f)\approx10^{-20}\|f\|^2$ and the
individual blocks are of order $\|f\|^2$, double-precision evaluation
of the geometric side is meaningless.  This elementary observation,
quantified in Section~\ref{sec:failures}, once suggested that upper
bounds had to be computed on the zeros side.  Multiprecision interval
arithmetic removes the obstruction: Section~\ref{sec:pipeline}
certifies geometric-side Rayleigh quotients down to
$10^{-283}$, matching (and typically slightly sharpening) the
zeros-side proposals.  Lower bounds continue to require one-sided,
sign-definite estimates on the geometric side.
\end{remark}

\begin{remark}[Variational upper bounds]\label{rem:variational}
For any finite-dimensional subspace $V$ of admissible test functions,
the minimal eigenvalue of the restriction of $Q$ to $V$ (with respect
to the $L^2$ Gram matrix) is an upper bound for $\lmin(L)$.  All
numbers in Table~\ref{tab:main} are of this form; no claim of
convergence to the true infimum is needed for the upper bounds to be
rigorous.
\end{remark}

Positivity of the infimum under RH needs a proof: in an
infinite-dimensional space, $Q(f)>0$ for every $f\ne0$ does not by
itself exclude $\inf_{\|f\|=1}Q(f)=0$.  What rules this out is a
sampling inequality for the Paley--Wiener space, available because the
zeros are, conditionally, dense at every sufficiently large height.

\begin{proposition}[The infimum is positive under RH]\label{prop:positive}
Assume RH.  For every $L>0$ there is $c_L>0$ such that
$Q(f)\ge c_L\|f\|_2^2$ for all admissible $f$ supported in $[-L,L]$;
in particular $\lmin(L)>0$.
\end{proposition}

\begin{proof}
For such $f$, the transform $\fhat$ is even and belongs to the
Paley--Wiener space $PW_L^2$ of entire functions of exponential type
at most $L$ that are square-integrable on $\R$.  By Littlewood's
theorem (a consequence of RH; see \cite[Ch.~XIII]{Titchmarsh1986}) the
gaps between consecutive ordinates satisfy
$\gamma_{n+1}-\gamma_n\le\eta(\gamma_n)$ with $\eta(t)\downarrow0$, so
the symmetrized ordinate set $\{\pm\gamma\}$ is $\eta(T_1)$-dense in
$\{|t|\ge T_1\}$ for every $T_1$.  Fix $\delta=\pi/(4L)$ and choose
$T_1$ with $\eta(T_1)\le\delta/2$; let $\Lambda$ be a maximal
$\delta$-separated subset of $\{\pm\gamma:\gamma\ge T_1\}$.  By
maximality, consecutive points of $\Lambda$ on either half-line lie
within $\delta+2\eta(T_1)\le2\delta$ of each other, so every interval
of length $r$ contains at least $(r-2T_1)/(2\delta)$ points of
$\Lambda$; hence the lower uniform (Beurling) density is
$D^{-}(\Lambda)\ge\frac{1}{2\delta}=\frac{2L}{\pi}>\frac{L}{\pi}$.  A
separated set of lower uniform density exceeding $L/\pi$ is a set of
sampling for $PW_L^2$ (Beurling's sampling theorem; see
\cite[Ch.~2--3]{Seip2004}): there exists $c>0$, depending only on $L$,
$\delta$ and $T_1$, with
$\sum_{\lambda\in\Lambda}|\fhat(\lambda)|^{2}\ge c\,\|\fhat\|_{2}^{2}$.
Since $\fhat$ is even and $\Lambda$ is contained in the symmetrized
ordinate set,
\[
  Q(f)\;=\;\sum_{\gamma>0}2\,|\fhat(\gamma)|^{2}
  \;\ge\;\sum_{\lambda\in\Lambda}|\fhat(\lambda)|^{2}
  \;\ge\;c\,\|\fhat\|_{2}^{2}
  \;=\;2\pi c\,\|f\|_{2}^{2}. \qedhere
\]
\end{proof}

\begin{remark}
The proposition is where RH enters twice: once through the zeros-side
identity \eqref{eq:zerosside}, and once through Littlewood's gap bound.
Unconditionally the gaps are only known to be bounded, which yields a
separated subset of some fixed density, and that is insufficient once
$L/\pi$ exceeds it.  No effective $c_L$ comes out of the argument.  An
effective, unconditional value of $c_L$ is exactly what
Theorem~\ref{thm:L08} provides, and the true decay of $c_L=\lmin(L)$
is what Sections~\ref{sec:results}--\ref{sec:theorem} measure and
bound.
\end{remark}

\subsection{Validation of the geometric-side assembly}
\label{sec:validation}

Convention errors are the dominant practical risk in computations of
this kind, so we record the end-to-end validation performed on the
lower-bound implementation.  The matrix of $Q$ in the basis
\eqref{eq:Tn} was assembled twice, (a) in the frequency domain from
\eqref{eq:symbol} by certified quadrature, and (b) in the time domain
via the exact archimedean identity of Lemma~\ref{lem:arch} below.
Separately, the bottom eigenvector $v$ of the arithmetic-side matrix
was evaluated on the zero side, as $2\sum_{j\le J}F_v(\gamma_j)^2$
over the first $J=400$ ordinates of $\zeta$.  At $L=0.5,0.8,1.0$ the
two arithmetic-side assemblies agree to all displayed digits, and the
zero-side partial sums are consistent with them in the only way they
can be, lying a few percent below the floors, since the missing tail
over $\gamma_j>\gamma_{400}$ is a sum of squares:
\[
\begin{array}{c|cc}
L & \lammin(\text{arith.\ side}) & 2\sum_{j\le400}F_v(\gamma_j)^2\\
\hline
0.5 & 9.4039\times10^{-7} & 9.1610\times10^{-7}\\
0.8 & 5.7499\times10^{-16} & 5.3290\times10^{-16}\\
1.0 & 1.6236\times10^{-20} & 1.5922\times10^{-20}
\end{array}
\]
These validation runs use a deliberately small basis ($16$ even
Legendre modes), so the floors in the table sit above the converged
window values quoted elsewhere (e.g.\ $1.66\times10^{-17}$ at
$L=0.8$); what is being validated is the normalization and the
zeros-side identity, not the floor itself.  The agreement confirms
the normalization of \eqref{eq:Q} at the $10^{-20}$
scale, twenty digits past the leading cancellation.  It also cross-links the
two halves of the paper: the same zeros-side functional evaluated here
for validation is the one whose infimum the upper-bound pipeline of
Section~\ref{sec:pipeline} certifies.

\begin{lemma}[Exact time-domain archimedean term]\label{lem:arch}
Let $g$ be real-valued, even and continuous with
$\supp g\subseteq[-2L,2L]$, suppose the one-sided limit
\begin{equation}\label{eq:archcorner}
  c\;=\;\lim_{x\to0^+}\frac{g(0)-g(x)}{x}
\end{equation}
exists, and suppose
$\int_{\R}|\widehat g(t)|\log(2+|t|)\,dt<\infty$.  Then, with $\gamma$
Euler's constant,
\begin{align*}
\frac1{2\pi}\int_{\R}\widehat g(t)\Bigl[\Rey\,\psi\bigl(\tfrac14
+\tfrac{it}2\bigr)-\log\pi\Bigr]dt
&=-\bigl(\gamma+\log\pi+\log(1-e^{-4L})\bigr)g(0)\\
&\quad+\int_0^{2L}
\frac{2\,[\,e^{-2x}g(0)-e^{-x/2}g(x)\,]}{1-e^{-2x}}\,dx ,
\end{align*}
both sides converging absolutely.  The integrand of the last integral
extends continuously to $x=0$ with value $c-\tfrac32g(0)$, and it
inherits on $[0,2L]$ whatever regularity the restriction
$g|_{[0,2L]}$ has there.  In the assembly of
Section~\ref{sec:validation} that restriction is a polynomial, so the
extended integrand is real-analytic on $[0,2L]$; note that $c\ne0$ in
general, the even extension of $g$ having a corner at the origin.
\end{lemma}

\begin{proof}
\emph{The kernel.}  Gauss's representation
$\psi(z)=-\gamma+\int_0^1\frac{1-s^{z-1}}{1-s}\,ds$ is valid for
$\Rey\,z>0$, hence at $z=\frac14+\frac{it}2$, where
$s^{z-1}=s^{-3/4}e^{i(t/2)\log s}$.  Taking real parts and substituting
$s=e^{-2x}$, so that $ds=-2e^{-2x}\,dx$, $s^{-3/4}=e^{3x/2}$ and
$\frac t2\log s=-tx$,
\begin{equation}\label{eq:archkernel}
  \Rey\,\psi\bigl(\tfrac14+\tfrac{it}2\bigr)
  \;=\;-\gamma+\int_0^\infty k(x)
  \bigl[e^{-2x}-e^{-x/2}\cos(tx)\bigr]dx,
  \qquad k(x)=\frac2{1-e^{-2x}} .
\end{equation}

\emph{Size of the kernel.}  Split the bracket as
$(e^{-2x}-e^{-x/2})+e^{-x/2}(1-\cos(tx))$.  The first part does not
involve $t$; since $k(x)=x^{-1}+O(1)$ and
$e^{-2x}-e^{-x/2}=-\frac32x+O(x^2)$ as $x\to0^+$, while $k$ is bounded
and $e^{-2x}-e^{-x/2}=O(e^{-x/2})$ as $x\to\infty$, it contributes
$O(1)$.  The second part is non-negative; split it at $x=1/|t|$.  On
$(0,1/|t|)$ use $1-\cos(tx)\le t^2x^2/2$ and $k(x)\le Cx^{-1}$, giving
$O(1)$; on $(1/|t|,1)$ use $1-\cos(tx)\le2$ and $k(x)\le Cx^{-1}$,
giving $O(\log(2+|t|))$; on $(1,\infty)$ use $k=O(1)$ against the
factor $e^{-x/2}$, giving $O(1)$.  Hence
\begin{equation}\label{eq:archsize}
  \int_0^\infty k(x)\bigl|e^{-2x}-e^{-x/2}\cos(tx)\bigr|dx
  \;=\;O\bigl(\log(2+|t|)\bigr) .
\end{equation}
In particular \eqref{eq:archkernel} converges absolutely and
$\Rey\,\psi(\frac14+\frac{it}2)-\log\pi=O(\log(2+|t|))$, so the
left-hand side of the lemma converges absolutely by the hypothesis on
$\widehat g$.

\emph{Exchange of integrals.}  By \eqref{eq:archsize} and that same
hypothesis,
\[
  \int_{\R}|\widehat g(t)|\int_0^\infty k(x)
  \bigl|e^{-2x}-e^{-x/2}\cos(tx)\bigr|dx\,dt\;<\;\infty ,
\]
so Fubini's theorem applies.  Since $\widehat g\in L^1$ and $g$ is
continuous, real and even, Fourier inversion gives
$\frac1{2\pi}\int_{\R}\widehat g(t)\,dt=g(0)$ and
$\frac1{2\pi}\int_{\R}\widehat g(t)\cos(tx)\,dt=g(x)$ at every $x$.
Substituting \eqref{eq:archkernel} and exchanging the order of
integration therefore yields
\begin{equation}\label{eq:archhalf}
  \frac1{2\pi}\int_{\R}\widehat g(t)
  \Bigl[\Rey\,\psi\bigl(\tfrac14+\tfrac{it}2\bigr)-\log\pi\Bigr]dt
  =-(\gamma+\log\pi)g(0)
  +\int_0^\infty k(x)\bigl[e^{-2x}g(0)-e^{-x/2}g(x)\bigr]dx .
\end{equation}

\emph{The tail in closed form.}  On $[2L,\infty)$ we have $g\equiv0$, so
there the integrand of \eqref{eq:archhalf} is $k(x)e^{-2x}g(0)$, and
the substitution $u=e^{-2x}$ gives
\[
  \int_{2L}^\infty\frac{2e^{-2x}}{1-e^{-2x}}\,dx
  \;=\;\int_0^{e^{-4L}}\frac{du}{1-u}
  \;=\;-\log\bigl(1-e^{-4L}\bigr) .
\]
Splitting \eqref{eq:archhalf} at $2L$ and inserting this value gives the
identity of the lemma.

\emph{Behaviour at the origin.}  By \eqref{eq:archcorner},
$g(x)=g(0)-cx+o(x)$ as $x\to0^+$, so
\[
  e^{-2x}g(0)-e^{-x/2}g(x)
  =g(0)\bigl(e^{-2x}-e^{-x/2}\bigr)+e^{-x/2}\bigl(g(0)-g(x)\bigr)
  =\bigl(c-\tfrac32g(0)\bigr)x+o(x),
\]
while $1-e^{-2x}=2x+O(x^2)$.  The quotient therefore tends to
$c-\frac32g(0)$.  Moreover $2\bigl[e^{-2x}g(0)-e^{-x/2}g(x)\bigr]$ and
$1-e^{-2x}$ are, on $[0,2L]$, as regular as $g|_{[0,2L]}$, and the
latter has the simple zero $2x+O(x^2)$ at the origin, which the former
matches; the quotient therefore inherits that regularity, and is
real-analytic when $g|_{[0,2L]}$ is.
\end{proof}

Lemma~\ref{lem:arch} eliminates every oscillatory or singular integral
from the assembly of $Q$: prime terms are shifts $g(\pm\log n)$, the
pole term is a rank-one $\cosh$ moment, and the archimedean term is
the integral above, whose integrand becomes real-analytic on $[0,2L]$
once its removable singularity at the origin is filled in.  All matrix
entries therefore converge geometrically under Gauss quadrature and are
certifiable at any precision.  We use this assembly for validation and floor curves; the
certificates of Sections~\ref{sec:certified}--\ref{sec:L119} use the
frequency-side assembly, whose certification is Lemma~\ref{lem:quad}.

\section{Envelope lemmas}\label{sec:envelope}

\begin{lemma}[Crude envelope]\label{lem:envelope}
For all $t\ge\frac{15}4$,
\[
\Rey\,\psi\Bigl(\tfrac14+\tfrac{it}2\Bigr)-\log\pi
\;\ge\;\log\frac t{2\pi}-\frac1t ,
\qquad\text{hence}\qquad
\Psi_L(t)\;\ge\;\log\frac t{2\pi}-\frac1t-A_L .
\]
\end{lemma}

\begin{proof}
By $\psi(z)=\psi(z+1)-1/z$ it suffices to bound
$\psi\bigl(\tfrac54+\tfrac{it}2\bigr)$.  Binet's second formula gives,
for $\Rey\,z>0$,
$\psi(z)=\log z-\frac1{2z}-2\int_0^\infty
\frac{u\,du}{(u^2+z^2)(e^{2\pi u}-1)}$.
At $z=\frac54+\frac{it}2$ we have
$|u^2+z^2|\ge|\Im(z^2)|=\frac{5t}4$ for all $u\ge0$, and
$\int_0^\infty\frac{u\,du}{e^{2\pi u}-1}=\frac1{24}$, so the integral
term is $\le\frac{8}{5t}\cdot\frac1{24}=\frac1{15t}$.  Furthermore
$\Rey\log z\ge\log(t/2)$, $|\Rey\frac1{2z}|\le\frac{5}{2t^2}$, and the
shift term satisfies $|\Rey\frac1{1/4+it/2}|\le\frac1{t^2}$.  The
total deficit is $\le\frac1{15t}+\frac7{2t^2}\le\frac1t$ for
$t\ge\frac{15}4$.  For the comb, trivially
$\bigl|\sum c_{p,k}\cos(tk\log p)\bigr|\le A_L$.
\end{proof}

\begin{lemma}[Optimality of the crude constant]\label{lem:sharp}
Write $P_L(t)=\sum_{\log n<2L}\frac{2\Lambda(n)}{\sqrt n}\cos(t\log n)$
for the prime comb.  Then
\[
\sup_{t\in\R}P_L(t)\;=\;A_L .
\]
Consequently the crude constant $C=A_L$ of Lemma~\ref{lem:envelope} is
already optimal among all pointwise bounds $P_L(t)\le C$ valid for
every $t$, which is the only information about the comb that the
reduction of Theorem~\ref{thm:reduction} uses.
\end{lemma}

\begin{proof}
$P_L(t)\le A_L$ is trivial (every cosine is $\le1$ and the
coefficients are positive), with equality at $t=0$.  That the supremum
is approached at arbitrarily large $t$ as well follows from Weyl
equidistribution.  The numbers $\{\log p\}_p$ are linearly independent
over $\mathbb Q$ by unique factorization, so the phases
$(t\log p_1,\dots,t\log p_r)$ equidistribute mod $2\pi$ on the full
torus.  Hence $t$ can be chosen so that every $\theta_p=t\log p$, and
therefore every $k\theta_p$, lies simultaneously arbitrarily close to
$0$, where all cosines are arbitrarily close to $+1$.
\end{proof}

\begin{remark}[A tempting sharpening, and why it fails]
\label{rem:nosharp}
Grouping the comb by primes, with
$c_{p,k}=2\Lambda(p^k)/p^{k/2}$, the quantities
\[
\mu_p\;=\;-\min_{\theta\in[0,\pi]}\ \sum_{k:\,k\log p<2L}
c_{p,k}\cos(k\theta),
\qquad
\Aeff\;=\;\sum_{p:\,\log p<2L}\mu_p
\]
satisfy $\inf_tP_L(t)=-\Aeff$, by the same Weyl argument with the
phases sent to the per-prime block minimizers.  Here $\Aeff$ is often
much smaller than $A_L$, because the powers of $p$ share the phase
$\theta_p$ and cannot all reach $-1$ together; at $L=1.19$, for
instance, $\Aeff=4.6948$ against $A_L=7.0750$.  It is tempting to
substitute $\Aeff$ for $A_L$ in Lemma~\ref{lem:envelope}.  This is
wrong.  The inequality $P_L\ge-\Aeff$ bounds the comb from below, hence
the symbol $\Psi_L=H-P_L$ from above, whereas the envelope needs an
upper bound for $P_L$; and by Lemma~\ref{lem:sharp} no pointwise upper
bound below $A_L$ exists, since at $\theta_p=0$ all cosines align at
$+1$ and the per-prime grouping buys nothing.  What $\Aeff$ measures is
the depth of the wells of $\Psi_L$, not its high-frequency floor.  An
earlier draft of this work made exactly this substitution;
Section~\ref{sec:L119} discusses the resulting retracted
support-$2.38$ claim and the corrected cost, and
Section~\ref{sec:failures2} records the error.  The thresholds are:
\[
\begin{array}{c|cc}
L & A_L & T_1=2\pi e^{A_L}\\
\hline
0.8 & 2.9420 & 119\\
1.0 & 5.8525 & 2187\\
1.19 & 7.0750 & 7.4\times10^3\\
1.2 & 8.5210 & 3.2\times10^4\\
1.4 & 10.290 & 1.9\times10^5\\
1.6 & 14.323 & 1.0\times10^7\\
2.0 & 24.383 & 2.4\times10^{11}
\end{array}
\]
\end{remark}

\section{Proof of Theorem \ref{thm:reduction}}\label{sec:reduction}

Split $\int_0^\infty=\int_0^{\Tsharp}+\int_{\Tsharp}^\infty$ in
\eqref{eq:Q} (for even real $f$, $\int_0^\infty|F|^2=\pi\|f\|^2$ by
Parseval).  On $[\Tsharp,\infty)$,
Lemma~\ref{lem:envelope} gives
$\Psi_L(t)\ge\log\frac t{2\pi}-\frac1t-A_L\ge\beta^*$, the last step
because the middle expression is increasing in $t$.  Hence
\[
\frac1{2\pi}\int_{|t|\ge \Tsharp}|F|^2\Psi_L
\;\ge\;\frac{\beta^*}\pi\int_{\Tsharp}^\infty|F|^2
\;=\;\beta^*\Bigl(\|f\|^2-\frac1\pi\int_0^{\Tsharp}|F|^2\Bigr),
\]
which is \eqref{eq:R}.  For the localization: the operator $C$ with
symbol $(\Psi_L-\beta^*)\chi_{[0,\Tsharp]}$ has entries
$C_{nm}=\frac1\pi\int_0^{\Tsharp}(\Psi_L-\beta^*)\widehat T_n
\widehat T_m\,dt$, and by \eqref{eq:Tn} and the elementary bound
\begin{equation}\label{eq:forbidden}
|j_n(x)|\;\le\;\frac{x^n}{(2n+1)!!}\qquad(x\ge0),
\end{equation}
each row satisfies
$|C_{nm}|\le\max_{[0,\Tsharp]}|\Psi_L-\beta^*|\cdot
2\sqrt{L\nu_n}\,\frac{(\Tsharp L)^n}{(2n+1)!!}\cdot2\sqrt{L\nu_m}$,
which decays super-exponentially once $n>eL\Tsharp/2$ by Stirling.
The pole vector decays super-exponentially as well ($\cosh$ is entire:
its Legendre coefficients decay faster than any geometric rate).

It remains to pass from the leading block to the whole of $H$, and here
a coupling estimate, small but not zero, has to be stated explicitly.
Fix a cut after $N$ even modes, with first discarded order
$2N\gtrsim eL\Tsharp/2$, and write the matrix of $R$ in block form
\[
M_R\;=\;\begin{pmatrix}A&B\\ B^{\mathsf T}&D\end{pmatrix},
\]
with $A$ the leading $N\times N$ block, $D$ the tail block, and $B$
the leading--tail coupling.  Positive semidefiniteness of the two
diagonal blocks alone does not imply $M_R\succeq0$; the correct
statement is the elementary two-block bound
\begin{equation}\label{eq:twoblock}
\lammin(M_R)\;\ge\;\min\bigl(\lammin(A),\,\lammin(D)\bigr)
-\|B\| ,
\end{equation}
which follows from
$x^{\mathsf T}Ax+2x^{\mathsf T}By+y^{\mathsf T}Dy
\ge\lammin(A)|x|^2-2\|B\||x||y|+\lammin(D)|y|^2$ and
$2|x||y|\le|x|^2+|y|^2$.  The two tail quantities are controlled by
the decay above: the Gershgorin bound on the rows of $D$ gives
$\lammin(D)\ge\beta^*-\varepsilon_D$, and the spectral norm of the
rectangular coupling block is bounded by the Schur test
$\|B\|\le\bigl(\|B\|_1\|B\|_\infty\bigr)^{1/2}
=\bigl(\max_m\textstyle\sum_{n\ge N}|B_{nm}|
\cdot\max_{n\ge N}\sum_{m<N}|B_{nm}|\bigr)^{1/2}
\le\varepsilon_B$; both the row sums (at most $N$ terms, each
super-exponentially small) and the column sums (convergent by the
decay in $n$) are explicit sums of the entry bounds over rows
$n\ge N$ (in the certified run of Section~\ref{sec:certified}, both
constants are below $10^{-100}$).  With
\eqref{eq:twoblock}, $\lammin(A)\ge\lambda_0$ implies
$R\ge(\min(\lambda_0,\beta^*-\varepsilon_D)-\varepsilon_B)\|f\|^2$ on
all of $H$, which is the statement of the theorem.  \qed

\begin{remark}
The proof discards three quantities, each with a definite sign or an
explicit bound: the excess of $\Psi_L$ over $\beta^*$ beyond
$\Tsharp$, the tail deviation $\varepsilon_D$, and the coupling norm
$\varepsilon_B$.  This one-sidedness is what makes the certificate
robust, since finer quadrature or larger $N$ can only improve the
certified constant, never invalidate it.  Section~\ref{sec:synthesis}
measures the cost of the first discard: at $L=0.8$ it is less than one
order of magnitude.
\end{remark}

\section{The certified computation at support 1.6}
\label{sec:certified}

All computations use \texttt{mpmath} \cite{mpmath} with the
\texttt{gmpy2} backend.

\subsection{Assembly}
$L=0.8$, $\Tsharp=200$, envelope constant
$A_L=2.9419735\ldots$ (a finite sum, certified by direct
multiprecision evaluation), so
$\beta^*=\log\frac{200}{2\pi}-\frac1{200}-A_L=0.5134667\ldots$;
$N=200$ even Legendre orders $0,2,\dots,398$; dps $50$.  The
$C$-matrix is integrated over $[0,200]$ in $800$ panels of width
$\frac14$ with $32$-point Gauss--Legendre quadrature; spherical Bessel
values by Miller's backward recurrence; the pole vector by Gauss
quadrature of order $320$ (the integrand
$\bar P_n(x/L)\cosh(x/2)$ is entire, not polynomial, so ``exact'' is
not claimed: the same Bernstein-ellipse argument as in
Lemma~\ref{lem:quad} bounds the quadrature error at this order far
below the $10^{-40}$ level).
The assembled matrix is archived.

\subsection{Quadrature certification}

\begin{lemma}[Ellipse bound]\label{lem:quad}
On each panel the integrand of $C_{nm}$ extends analytically to the
Bernstein ellipse with sum of semiaxes $\rho=6.55$ relative to the
panel (the nearest singularities of $\Psi_L$ are the digamma poles at
$\Im t=\pm\frac12$; we use the strip $|\Im t|\le0.4$).  Here
``analytic extension'' means the digamma term
$\Re\,\psi(\tfrac14+\tfrac{it}2)$ is replaced by its half-sum
$\tfrac12\bigl[\psi(\tfrac14+\tfrac{it}2)+\psi(\tfrac14-\tfrac{it}2)\bigr]$,
which is analytic in $t$ and agrees with it on the real axis.  On that
ellipse, by the Poisson representation
$j_n(z)=\frac{(-i)^n}2\int_{-1}^1P_n(s)e^{izs}ds$,
\[
|j_n(z)|\le e^{|\Im z|}
\quad\Rightarrow\quad
|\widehat T_n\widehat T_m|\le 4L\sqrt{\nu_n\nu_m}\,e^{2L\cdot0.4},
\]
and $|\Psi_L-\beta^*|\le20$ there; hence the $32$-point Gauss error
per panel is at most $4M\rho/((\rho-1)(\rho^{64}-1))$ with
$M\le4.9\times10^4$, i.e.\ $\le1.3\times10^{-47}$, and the total
per-entry quadrature error is $\varepsilon_Q\le1.03\times10^{-44}$.
The error matrix has spectral norm
$\le c_{\rm err}:=N\varepsilon_Q\le2.06\times10^{-42}$.
\end{lemma}

\subsection{Tail and coupling certification}
At the cut order $400$: by \eqref{eq:forbidden},
$\max_{t\le200}|\widehat T_{400}(t)|\le4.3\times10^{-108}$ (at
$\Tsharp=150$: $4.5\times10^{-158}$), and the pole tail is
$<10^{-390}$.  The row and column sums built from these entry bounds
give both tail constants of \eqref{eq:twoblock}: the Gershgorin tail
diagonals exceed $\beta^*-10^{-100}$, i.e.\
$\varepsilon_D\le10^{-100}$, and both factors in the Schur bound
$\|B\|\le(\|B\|_1\|B\|_\infty)^{1/2}$ are dominated by the same
super-exponential bounds, giving $\varepsilon_B\le10^{-100}$.  Orders
beyond $400$ therefore cost a total of $2\times10^{-100}$ against the
certified constant.

\subsection{Verified positive semidefiniteness}

\begin{lemma}[Cholesky residual]\label{lem:psd}
Let $M'$ be symmetric and suppose a floating-point Cholesky
factorization produces $\tilde L$ with
$\|M'-\tilde L\tilde L^{\mathsf T}\|_\infty\le r$ (residual computed
in the same arithmetic, with product-rounding slack $s$).  Then
$\lammin(M')\ge-(r+s)$.
\end{lemma}

\begin{proof}
$\tilde L\tilde L^{\mathsf T}\succeq0$ and Weyl's inequality.
\end{proof}

The error budget must distinguish the exact leading block $A$, defined
by the integrals, from the computed matrix $\tilde A$;
Lemma~\ref{lem:quad} gives $\|A-\tilde A\|\le c_{\rm err}$.  Applying
Lemma~\ref{lem:psd} to $\tilde A-(\lambda_0+c_{\rm err})I$ with
$\lambda_0=9\times10^{-18}$: the factorization succeeds, with
$r=1.06\times10^{-50}$ and $s=3.6\times10^{-43}$, so
$\lammin(\tilde A)\ge\lambda_0+c_{\rm err}-(r+s)$ and hence, by
Weyl's inequality,
\[
\lammin(A)\;\ge\;\lammin(\tilde A)-c_{\rm err}
\;\ge\;\lambda_0-(r+s)
\;\ge\;9\times10^{-18}-4\times10^{-43}.
\]
Feeding this into Theorem~\ref{thm:reduction} (with
$\varepsilon_D,\varepsilon_B\le10^{-100}$ from the previous
subsection, and $\beta^*-\varepsilon_D\gg\lambda_0$) gives
\[
Q(f)\;\ge\;\bigl(\lambda_0-(r+s)-\varepsilon_B\bigr)\|f\|_2^2
\;\ge\;8.9\times10^{-18}\,\|f\|_2^2 ,
\]
which proves Theorem~\ref{thm:L08} with an explicit positive margin
of $10^{-19}$ to spare. \qed

\subsection{Independent checks}
(a) An earlier, independent certification at $\Tsharp=150$
($\beta^*=0.2241$) gives $\lammin\ge1.2\times10^{-18}$.
(b) Monotonicity: $R_{\Tsharp}$ increases pointwise in $\Tsharp$, and
the certified constants respect it, as does the window floor:
$1.2\times10^{-18}\le\lammin(R_{150})$,
$9\times10^{-18}\le\lammin(R_{200})$, and both lie below
$\lammin(Q|_{\rm window})=1.656\times10^{-17}$, the last value from
the exact time-domain assembly of Section~\ref{sec:validation}.
(c) The reference (non-certified) spectrum of $R_{150}$ has bottom
eigenvalues $1.356\times10^{-18},\,2.32\times10^{-12},\,
2.9\times10^{-7},\,2.1\times10^{-3},\,\beta^*,\dots$ (the first two by
inverse iteration on the stored matrix, residuals $2.4\times10^{-37}$
and $1.1\times10^{-32}$; an earlier draft quoted $1.78\times10^{-18}$
for the first, a stale value which a failed Cholesky at shift
$1.5\times10^{-18}$ already contradicted): four collective
near-null directions sit above the floor, and the bottom eigenvector
concentrates on Legendre orders $0$--$10$, carrying mass
$5\times10^{-76}$ at orders $\ge300$, so the truncation is converged
with a very large margin.
(d) Externally, the certified upper bound
$\lmin(0.8)\le2.27\times10^{-17}$ of Table~\ref{tab:main}, obtained
unconditionally on the geometric side by the pipeline of
Section~\ref{sec:geocert}, sits consistently above the certified
lower bound and within $40\%$ of the measured window floor
$1.656\times10^{-17}$ (the residual difference is attributable to the
boundary condition $f(\pm L)=0$ of the sine basis used there, harmless
for upper bounds).

\section{Parity of the window ground state}\label{sec:parity}

Everything so far concerns even test functions.  This section resolves
the odd sector at support $1.6$ and, with it, a spectral question
recently isolated by Connes--Consani--Moscovici \cite{CCMZeta2025}.
Their construction realizes approximations to the zeta zeros from the
ground state of the semilocal Weil form at scale $\lambda$, and its
theoretical foundation, the zero-localization theorem of
Connes--van Suijlekom \cite{ConnesVanSuijlekom2025}, applies under the
hypothesis that the bottom of the spectrum is a \emph{simple}
eigenvalue whose eigenfunction is \emph{even} (invariant under
$u\mapsto u^{-1}$).  They list this as the first missing step of the
program; an independent reproduction has verified it empirically at
many scales, but not rigorously at any.  We certify it at
$\lambda=e^{0.8}$, where the semilocal form has prime content
$\{2,3,4\}$ and coincides, place by place, with the window form
studied in this paper (up to the normalization dictionary, whose
line-by-line verification against \cite{CCMZeta2025} we have not
carried out; the statements below, being ordering statements, are
invariant under positive rescaling of the form).

\subsection{Parity decoupling}

\begin{lemma}[Parity splitting]\label{lem:parity}
Let $f=f_e+f_o$ be real with $\supp f\subseteq[-L,L]$, split into even
and odd parts.  Then
\[
Q(f)\;=\;Q(f_e)+Q(f_o),
\]
where $Q(f_e)$ has the pole term $+2\bigl(\int f_e\cosh(x/2)\,dx\bigr)^2$
of \eqref{eq:Q}, while
\[
Q(f_o)\;=\;-\,2\Bigl(\int f_o\sinh(x/2)\,dx\Bigr)^{\!2}
+\frac1{2\pi}\int_{\R}|F_o(t)|^2\,\Psi_L(t)\,dt :
\]
\emph{the pole term changes sign in the odd sector}.  Moreover, for
complex $f$, $Q(f)=Q(\Rey f)+Q(\operatorname{Im} f)$.
\end{lemma}

\begin{proof}
With $g=f\star\tilde f$ one has $\hat g(z)=\hat f(z)\hat f(-z)$, so the
pole term is $\hat g(i/2)+\hat g(-i/2)=2\hat f(i/2)\hat f(-i/2)$.
Writing $\hat f(\mp i/2)=\int f e^{\pm x/2}\,dx=c\pm s$ with
$c=\int f_e\cosh(x/2)$, $s=\int f_o\sinh(x/2)$ gives
$2(c^2-s^2)$: no cross term, and the stated signs.  In the multiplier
term, $F_e$ is real even and $F_o$ is imaginary odd, so
$|F|^2=|F_e|^2+|F_o|^2$ pointwise.  For the complex statement, the
cross term of the multiplier is $\operatorname{Im}(\overline{F_r}F_i)$, odd in $t$,
and integrates to zero against the even symbol; the cross terms of the
two poles cancel in pairs.
\end{proof}

The sign reversal was validated independently of this computation: for
a random odd polynomial, the time-domain pole term
$2\int g(y)\cosh(y/2)\,dy$ computed by direct autocorrelation agrees
with $-2(\int f\sinh(x/2)\,dx)^2$ to relative error $2\times10^{-11}$
(script \texttt{pwl\_parity\_recon.py}).

The one-stroke reduction (Theorem~\ref{thm:reduction}) holds verbatim
in the odd sector: the inequality \eqref{eq:R} only discards the
positive multiplier tail beyond $\Tsharp$, which is parity-blind, and
the pole term (of either sign) sits inside the retained matrix.  The
tail and coupling bounds of Section~\ref{sec:reduction} are unchanged
(the odd pole tail obeys the same forbidden-region bound, with $\sinh$
majorized by $\cosh$).

\subsection{Certified constants}

The odd-sector reduced matrix at $\Tsharp=150$
($\beta^*=0.2241\ldots$, $200$ odd Legendre modes $1,3,\dots,399$,
$50$-digit assembly, quadrature budget
$c_{\rm err}=1.6\times10^{-42}$) has
\[
\lammin(M_{\rm odd})=9.1183\times10^{-15}
\quad\text{(inverse iteration, residual }2.3\times10^{-30}),
\]
and the verified Cholesky residual (Lemma~\ref{lem:psd}) at shift
$8.2065\times10^{-15}$ succeeds with $r=7.9\times10^{-51}$, so
\begin{equation}\label{eq:oddlower}
Q(f)\;\ge\;8.2065\times10^{-15}\,\|f\|_2^2
\qquad\text{for all real odd }f,\ \supp f\subseteq[-0.8,0.8].
\end{equation}

Upper bounds require the exact form rather than the reduced one (the
reduction discards a positive term).  Both integrals of the exact
assembly live on compact intervals: the archimedean identity of
Section~\ref{sec:validation} and shift-correlation integrands that are
polynomials of degree $\le239$, integrated exactly by Gauss--$160$.
Evaluating the exact Rayleigh quotient at the (truncated) bottom
vectors of the reduced matrices gives, for the window infima of the
two sectors,
\[
\lambda_1^{\rm even}\le2.5223\times10^{-16},
\qquad
\lambda_1^{\rm odd}\le2.3470\times10^{-14}.
\]
As an independent check, the zeros-side partial sums
$2\sum_{\gamma\le\gamma_{300}}F(\gamma)^2$ on the same vectors are
$1.767\times10^{-16}$ and $2.071\times10^{-14}$: below the Rayleigh
values, as they must be, by $12$--$30\%$, the right size for the
missing tail above $\gamma_{300}\approx541$.  (The zeros enter only
this check, not the certificates.)

For the even-sector gap, apply the Courant--Fischer principle: if the
restriction of $M_{\rm even}-s_2I$ to \emph{any} subspace of
codimension one is positive semidefinite, then
$\lambda_2(M_{\rm even})\ge s_2$.  Taking the orthogonal complement of
the approximate ground vector (an explicit Householder frame; no
orthogonality of the computed frame is assumed, full rank suffices)
and running the verified Cholesky on the restricted matrix at
$s_2=2.0857\times10^{-12}$ succeeds with residual $1.1\times10^{-50}$
and frame-multiplication slack $5.1\times10^{-43}$; the reference
value from deflated inverse iteration is
$\lambda_2(M_{\rm even})=2.3175\times10^{-12}$ (residual
$1.1\times10^{-32}$).  The lift from the $200$-mode head to the full
window form is the block-perturbation argument of
Theorem~\ref{thm:reduction}, at cost below $10^{-100}$.

\subsection{The parity theorem}

\begin{theorem}[Simple even ground state, support $1.6$]
\label{thm:parity}
Let $\lambda_1^{\rm even},\lambda_2^{\rm even}$ denote the first two
min--max values of $Q(f)/\|f\|_2^2$ over real even $f$ with
$\supp f\subseteq[-0.8,0.8]$, and $\lambda_1^{\rm odd}$ the infimum
over real odd $f$.  Then
\[
8.9\times10^{-18}\;\le\;\lambda_1^{\rm even}\;\le\;2.523\times10^{-16},
\qquad
\lambda_2^{\rm even}\;\ge\;2.085\times10^{-12},
\]
\[
8.206\times10^{-15}\;\le\;\lambda_1^{\rm odd}\;\le\;2.347\times10^{-14}.
\]
In particular the bottom of the window spectrum lies in the even
sector with clearance at least a factor $32$ against the odd sector
and $8\times10^{3}$ against the second even value: the ground state
is simple, and even.
\end{theorem}

\begin{corollary}[Positivity for arbitrary test functions]
\label{cor:allf}
For every complex $f\in L^2(\R)$ with $\supp f\subseteq[-0.8,0.8]$,
\[
Q(f)\;\ge\;8.9\times10^{-18}\,\|f\|_2^2 .
\]
Equivalently, the Weil functional is nonnegative on all
$g=f\star\tilde f^*$ with $\supp g\subseteq[-1.6,1.6]$, with no parity
or reality restriction on $f$.
\end{corollary}

\begin{proof}
Lemma~\ref{lem:parity}, Theorem~\ref{thm:L08} on the even parts,
\eqref{eq:oddlower} on the odd parts.
\end{proof}

Two remarks.  First, the certificate is geometric-side only; unlike
the upper-bound pipeline of Section~\ref{sec:pipeline}, nothing here
depends on zero data.  Second, the parity asymmetry is systematic.
Table~\ref{tab:parity} traces both sector floors across
$L\in[0.5,1.2]$ by converged exact assembly (reference values, not
certificates; obtained in the Legendre basis per sector, with an
$(N{-}12)$-mode convergence check at every entry).  Three features
stand out.  (i) The even floors sit $20$--$40\%$ below the certified
sine-basis upper bounds of Table~\ref{tab:main} at every $L$, as they
must.  (ii) The bottom of the window spectrum alternates parity,
$\lambda_1^{\rm even}<\lambda_1^{\rm odd}<\lambda_2^{\rm even}
<\lambda_2^{\rm odd}$ at every scanned $L$, the same alternation as
the prolate eigenfunctions of time--band limiting.  (iii) The sector
ratio $\lambda_1^{\rm odd}/\lambda_1^{\rm even}$ grows almost exactly
geometrically, $\approx10^{2.07\,L+1.32}$ over the scanned range
(residuals of $\log_{10}$ below $0.06$).  Its mechanism is \emph{not}
the pole sign, which in fact favors the odd sector (the pole term
adds $+2c^2$ to the even form and $-2s^2$ to the odd one); on the
zeros side the asymmetry must instead reflect the constraint
$F_o(0)=0$, which denies odd test functions the zero-free frequency
interval $[0,\gamma_1)$ in which even minimizers park most of their
mass.  We leave its quantitative law, and whether the odd sector
obeys a shifted version of \eqref{eq:law}, to a separate study.  The
second missing step of \cite{CCMZeta2025}, convergence of the
approximate ground-state zeros to the zeta zeros as the scale grows,
is untouched by the present certificate.

\begin{table}[ht]
\centering
\caption{Sector floors of the window form by converged exact assembly
(reference values, not certificates).  At $L=0.8$ the two enclosures
certified in this section,
$\lambda_1^{\rm even}\in[8.9\times10^{-18},2.52\times10^{-16}]$ and
$\lambda_1^{\rm odd}\in[8.21\times10^{-15},2.35\times10^{-14}]$,
contain the corresponding entries.  The last row is converged only to
within a factor $\approx2$ ($80$ modes per sector).}
\label{tab:parity}
\begin{tabular}{lllll}
\hline
$L$ & $\lambda_1^{\rm even}$ & $\lambda_1^{\rm odd}$ &
ratio & $\lambda_2^{\rm even}$ \\
\hline
$0.5$ & $9.34\times10^{-7}$  & $1.94\times10^{-4}$  & $208$  & $1.80\times10^{-2}$ \\
$0.6$ & $1.61\times10^{-9}$  & $5.97\times10^{-7}$  & $371$  & $1.03\times10^{-4}$ \\
$0.7$ & $4.18\times10^{-13}$ & $2.56\times10^{-10}$ & $612$  & $8.65\times10^{-8}$ \\
$0.8$ & $1.65\times10^{-17}$ & $1.57\times10^{-14}$ & $949$  & $8.38\times10^{-12}$ \\
$0.9$ & $4.14\times10^{-23}$ & $7.22\times10^{-20}$ & $1746$ & $7.25\times10^{-17}$ \\
$1.0$ & $5.88\times10^{-30}$ & $1.49\times10^{-26}$ & $2535$ & $2.18\times10^{-23}$ \\
$1.1$ & $2.04\times10^{-38}$ & $7.68\times10^{-35}$ & $3769$ & $1.54\times10^{-31}$ \\
$1.2$ & $7.94\times10^{-49}$ & $4.76\times10^{-45}$ & $5997$ & $1.46\times10^{-41}$ \\
\hline
\end{tabular}
\end{table}

\section{An exploratory computation at support 2.38, and a
retraction}\label{sec:L119}

An earlier draft of this work claimed a certified positivity theorem
at $L=1.19$ (support $2.38$).  The claim rested on substituting the
per-prime constant $\Aeff(1.19)=4.6948$ for $A_L(1.19)=7.0750$ in the
envelope, which lowered the apparent threshold to
$2\pi e^{\Aeff}=687$ and made the run $\Tsharp=1100$,
$\beta^*=0.46948$ seem sufficient.  As Remark~\ref{rem:nosharp}
explains, that substitution bounds the comb in the wrong direction and
is invalid.  By Lemma~\ref{lem:sharp} the true threshold is
$T_1=2\pi e^{A_L}\approx7.4\times10^3$, and a positive $\beta^*$ of
the same size requires $\Tsharp\approx1.2\times10^4$.  Since
$\Psi_L\ge\beta^*$ is not established on $[1100,\,1.2\times10^4]$, the
inequality $Q\ge R$ of \eqref{eq:R} is not proved for that run, and
nothing about $Q$ follows from it.  We retract the claim.

We nevertheless report the computation, for two reasons.  The first is
as data.  The assembled matrix ($N=950$ even orders to $1898$, dps
$70$, Gauss-$40$ on $2200$ panels of width $\frac12$, adaptive
forbidden-region cutoff, $10$ CPU-hours) is verified positive definite,
with a Cholesky shift ladder passing at
$10^{-60},10^{-55},10^{-52},10^{-50},10^{-48}$ and failing at
$10^{-46}$, so $\lammin\in(10^{-48},10^{-46})$, with residuals
$r=1.7\times10^{-70}$ and $s=2.9\times10^{-62}$ at the passing shift.
This is consistent with Weil positivity at support $2.38$, as under RH
it must be, and the scale $10^{-48}$ agrees with the law
\eqref{eq:law}, which predicts $\lmin(1.19)\approx10^{-47}$; but it
certifies nothing.  The second reason is methodological.  A first
attempt at dps $50$ could only conclude $\lammin<10^{-33}$, positive
definite but unresolvable, and it was the decay law of
Section~\ref{sec:results} that dictated the right response: raise the
precision to dps $70$, rather than distrust the sign.

\emph{The corrected cost of support $2.38$.}  The failure is real, not
merely proof-technical.  A direct scan of
$\log\frac t{2\pi}-\frac1t-P_L(t)$ (float64, step $10^{-3}$) shows
that on $[1100,\,2\pi e^{A_L+\beta^*}]$ this quantity drops below
$\beta^*$ on a set of measure $\approx18$, by as much as $2.0$ near
$t\approx1550$.  The comb has Lipschitz constant
$\sum_n 2\Lambda(n)\log n/\sqrt n\le11$ here, so a step of $10^{-3}$
can miss at most $0.011$, and the dip exceeds that by a factor of
nearly two hundred.  The Binet estimates behind
Lemma~\ref{lem:envelope} are moreover two-sided, the digamma term
equalling $\log\frac t{2\pi}$ to within $2/t\le2\times10^{-3}$ on this
range.  Hence $\Psi_L$ itself, and not merely its envelope, dips below
$\beta^*$ by about $2$ there, and $Q\ge R$ genuinely fails for the run
above.  The last such dip sits at $t_0\approx8.48\times10^3$, with
margin $\ge0.06$ from there to the crude-envelope takeover at
$1.19\times10^4$.  A valid one-stroke certificate at $L=1.19$
therefore needs $\Tsharp\gtrsim8.5\times10^3$ if the pointwise
envelope on $[\Tsharp,2\pi e^{A_L+\beta^*}]$ is replaced by a
certified interval-arithmetic lower bound for $\Psi_L$ itself (the
symbol is an explicit finite expression, so $\min_{[a,b]}\Psi_L$ is
certifiable by adaptive interval evaluation), or
$\Tsharp\approx1.2\times10^4$ with the envelope alone.  Either way it
needs $N\approx eL\Tsharp/2\approx1.4$--$2\times10^4$ Legendre orders
at a working precision near $90$ digits, the law placing the floor
near $10^{-47}$.  That is roughly three orders of magnitude more work
than the run above, out of reach for a pure-Python pipeline but not
for a compiled double-double assembly.  We have not carried it out.

\section{Synthesis: the two-sided enclosure}\label{sec:synthesis}

We combine the two halves here, before developing the upper-bound
pipeline in detail.  At $L=0.8$ the enclosure is
\[
  \underbrace{8.9\times10^{-18}}_{\text{Thm.~\ref{thm:L08}}}
  \;\le\;\lmin(0.8)\;\le\;
  \underbrace{2.27\times10^{-17}}_{\text{Table~\ref{tab:main}}} ,
\]
a certified two-sided determination of an RH-equivalent quantity to
within a factor $2.6$, \emph{with both inequalities unconditional}.
In particular $\inf\sigma(K_{0.8})\le2.27\times10^{-17}$ is an
operator fact, not an observation under RH.  Two conclusions follow.
First, the one-stroke reduction loses little.  The reduced form $R$
of \eqref{eq:R} discards the entire excess of $\Psi_L$ beyond
$\Tsharp$, yet still reaches to within an order of magnitude of the
true window floor: $\lammin(R_{150})=1.356\times10^{-18}$ against the
window value $1.66\times10^{-17}$, a factor $12$.  The loss is
governed by the artificial wells created by the frequency split,
that is, by the set $\{t\le\Tsharp:\Psi_L(t)<\beta^*\}$.  Second, and
conversely, the law \eqref{eq:law} calibrates the certification
effort: it predicts the scale of the bottom eigenvalue of a planned
certificate matrix, and hence the working precision required, an
interplay recorded in Section~\ref{sec:L119}.

\section{The resolution height \texorpdfstring{$\Tstar$}{T*} and the
  expected law}\label{sec:mechanism}

We now turn to the upper-bound half: how small does $\lmin(L)$
actually get?  A function supported in $[-L,L]$ has Fourier transform
of exponential type $L$; by Jensen's inequality its real zeros have
density at most $L/\pi$ in long windows.  The zeros of $\zeta$ have
density $\nu(t)=\frac{1}{2\pi}\ln\frac{t}{2\pi}$ at height $t$, which
reaches the budget $L/\pi$ exactly at
\[
  \Tstar(L)\;=\;2\pi e^{2L}.
\]
Below $\Tstar$ the test function can in principle interpolate, placing
a zero of $\fhat$ on every $\gamma_j$; above $\Tstar$ the zeta zeros
outnumber the available oscillations and $|\fhat|^2$ must leak onto
them.  One therefore expects the minimizer of \eqref{eq:lambdastar} to
suppress $|\fhat(\gamma)|$ for $\gamma<\Tstar$ down to the working
level $\sqrt{\lmin}$ and to pay its cost in the transition region
around $\Tstar$.  Figure~\ref{fig:portrait} confirms this picture in
detail.

How much does the suppression cost?  Killing $N=N(\Tstar)$ zeros with
an exactly exhausted budget is the regime of the Landau--Widom theorem
\cite{LandauWidom1980}: for time--band limiting to a window of area
$K$, the Shannon number, the eigenvalues of the concentration operator
plunge at the rate $\lambda_{K+j}\approx\exp\bigl(-\pi^{2}j/\ln K\bigr)$.
Modelling the suppression of $N$ zeros at saturated budget as an
excursion of order $N$ beyond the Shannon number of the window
suggests
\begin{equation}\label{eq:lwlaw}
  -\ln\lmin(L)\;\asymp\;\frac{\pi^{2}\,N(\Tstar)}{\ln N(\Tstar)}
  \times\text{const},
\end{equation}
a double-exponential decay in $L$ with a logarithmic correction, rather
than the naive ``one nat per killed zero'' guess
$-\ln\lmin\asymp N(\Tstar)$.  Our data decide between the two:
\eqref{eq:lwlaw} holds with overall constant $2\pi^{2}$, reproducing
the four asymptotic points to within $2.7\%$
(Section~\ref{sec:results}).

Note the two heights that now coexist: the resolution height
$\Tstar=2\pi e^{2L}$, which controls how far the minimizer must reach
and hence how small $\lmin$ becomes, and the comb-alignment threshold
$T_1=2\pi e^{A_L}$ with $A_L\sim4e^L$, which controls how far a
certificate must reach.  For $L\gtrsim1$ one has $T_1\gg\Tstar$, so
certifying positivity is exponentially more expensive than defeating
it.  This is the quantitative content of the barrier
(Theorem~\ref{thm:barrier}).

\section{Three generations of float64 experiments and how they fail}
\label{sec:failures}

We summarize the failure modes we encountered before moving to
multiprecision; details and scripts are in the repository
(Section~\ref{sec:repro}).  All three generations discretize
\eqref{eq:lambdastar} by a basis $(\varphi_k)$ of functions supported
in $[-L,L]$ and solve the generalized eigenproblem $Mc=\lambda Gc$
with $M_{jk}=Q(\varphi_j,\varphi_k)$ and $G$ the $L^2$ Gram matrix.

\textbf{Generation 1: hat functions, geometric side.}  Piecewise-linear
hat functions (uniform nodes, or arithmetic nodes at $\log p^k$)
assembled through the geometric side.  Two artifacts appeared.  First,
boundary half-hats make $f(\pm L)\ne0$, so $|\fhat(r)|^2$ decays only
like $r^{-2}$ and the truncated archimedean integral loses a positive
mass, producing spurious negative eigenvalues that are, at first
sight, structurally indistinguishable from a counterexample to RH.
Second, after fixing the boundary condition, the sweep produced a
convincing exponential law $\lmin(L)\sim e^{-2.64 L}$, which later
turned out to be an artifact of the float64 noise floor: the apparent
rate is the local slope of the true (much steeper) curve at the point
where it crosses $10^{-16}$.

\textbf{Generation 2: cubic $B$-splines and Richardson extrapolation.}
Replacing hats by $C^2$ cubic splines improves the spectral decay of
the basis from $r^{-2}$ to $r^{-4}$ per factor and demonstrably lowers
the variational bound at fixed dimension; a Richardson extrapolation
in the knot spacing then appeared to converge to a cleaner exponential
law with rate $1.52$.  The extrapolation, however, was converging to
the noise floor of the ambient discretization, not to $\lmin(L)$:
refining the mesh at fixed $L$ eventually increased the computed
minimum, violating the variational monotonicity that a Galerkin scheme
must obey.  We recommend this monotonicity check as a cheap
self-diagnostic for quadratic-form computations near the precision
floor.

\textbf{Generation 3: zeros side in float64.}  Evaluating
\eqref{eq:zerosside} term by term is cancellation-free, and certified
re-evaluation of the proposed minimizer gives genuine upper bounds
(e.g.\ $\lmin(0.7)\le5.5\times10^{-13}$, a valid bound lying within
$10\%$ of the later multiprecision value $5.00\times10^{-13}$).  But for
$L\gtrsim0.9$ the true values sink below $10^{-16}$ and float64
eigen-solvers return the noise floor: certified bounds stall near
$10^{-16}$ for all $L\in[1.0,2.2]$.  Figure~\ref{fig:history} displays
all three generations against the multiprecision curve.

\begin{figure}[t]
  \centering
  \includegraphics[width=\textwidth]{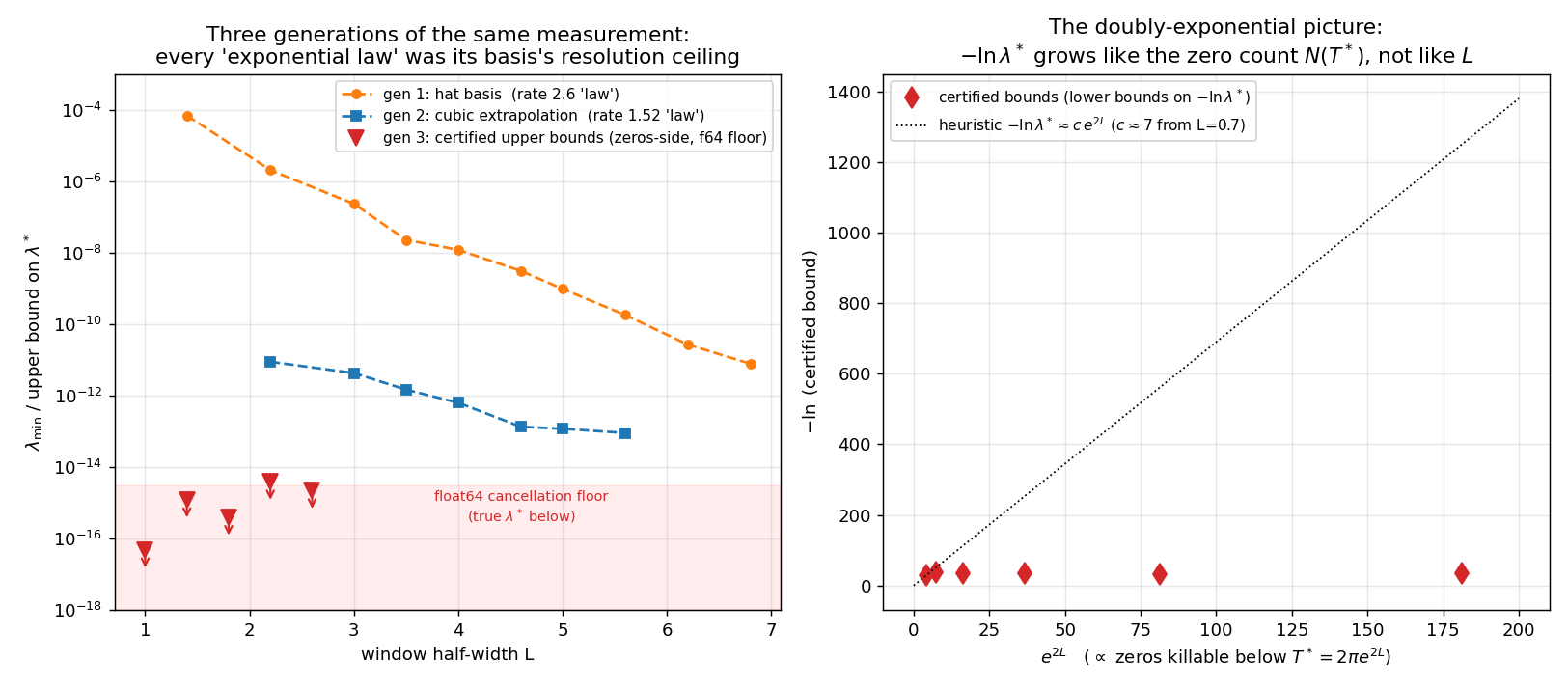}
  \caption{The three generations of float64 measurements.  Apparent
    exponential laws (rates $2.64$ and $1.52$) are local-slope
    artifacts; certified float64 bounds (squares) stall at the noise
    floor near $10^{-16}$.}
  \label{fig:history}
\end{figure}

\section{A certified multiprecision pipeline for upper bounds}
\label{sec:pipeline}

\subsection{Basis and block structure}

Fix $L$ and work with the orthogonal modes
\[
  s_k(u)\;=\;\sin\Bigl(\frac{k\pi(u+L)}{2L}\Bigr),\qquad
  u\in[-L,L],\quad k=1,\dots,K,
\]
which vanish at $\pm L$ and satisfy $\int s_js_k=L\,\delta_{jk}$, so
the Gram matrix is exactly $L\,I$ and the discretized problem is an
ordinary symmetric eigenproblem; no ill-conditioned mass matrix enters
at any precision.  The transforms are closed-form: with
$\alpha_k=k\pi/(2L)$,
\begin{equation}\label{eq:shat}
  \hat s_k(r)
  =\frac{\alpha_k\bigl(e^{-irL}-(-1)^k e^{irL}\bigr)}{\alpha_k^2-r^2}
  =\begin{cases}
    \dfrac{2\alpha_k\cos(rL)}{\alpha_k^2-r^2}, & k \text{ odd},\\[2ex]
    \dfrac{-2i\,\alpha_k\sin(rL)}{\alpha_k^2-r^2}, & k \text{ even},
  \end{cases}
\end{equation}
so the form \eqref{eq:zerosside} splits exactly into two parity blocks
(odd $k$: even real $\fhat$; even $k$: odd imaginary $\fhat$), halving
the dimension.  In all runs the minimizer lies in the odd-$k$ (even
$\fhat$) block; we verified this by solving both blocks for $L\le1.2$
and spot-checking at $L=1.4$.  We take
\[
K=\max\bigl(48,\ \lfloor 5L\Tstar/\pi\rfloor+8\bigr)
\]
and retain the odd $k<K$, so that $\alpha_{\max}\approx2.5\,\Tstar$
once $L\ge0.9$, the floor of $48$ being active (and the margin
correspondingly larger) for $L\le0.8$.  This rule reproduces the
dimension of each of the twelve runs reported below, from $24$ odd
modes at $L=0.5$ to $549$ at $L=2.0$.  Doubling the margin changes
$-\ln\lmin$ by $0.03\%$ at $L=1.6$ (Table~\ref{tab:checks}).

\subsection{Assembly and error budget}

For the odd block the discretized form reads
\[
  A_{jk}\;=\;\sum_{0<\gamma\le\Gamma_0} 8\cos^2(\gamma L)\,
  u_j(\gamma)u_k(\gamma)\;+\;T_{jk},
  \qquad u_k(\gamma)=\frac{\alpha_k}{\alpha_k^2-\gamma^2},
\]
with $\Gamma_0=1419.8$ (the first $1000$ zeros) and $T$ the tail
model of Definition~\ref{lem:tail}.  Every term is non-negative as a
quadratic form; the assembly is cancellation-free by construction.
The three error sources are controlled as follows.

\begin{lemma}[Zero-location error]\label{lem:zeros}
Let $\tilde\gamma_j=\gamma_j+\delta_j$ with $|\delta_j|\le\delta$ be the
ordinates used in the assembly, and let $M_j$ denote the supremum of
$|\fhat\,'|$ on the segment joining $\gamma_j$ to $\tilde\gamma_j$.
Then for every test function $f$,
\[
  \Bigl|\;\sum_j 2\bigl|\fhat(\tilde\gamma_j)\bigr|^{2}
        -\sum_j 2\bigl|\fhat(\gamma_j)\bigr|^{2}\;\Bigr|
  \;\le\;\sum_j\Bigl(4\,\bigl|\fhat(\gamma_j)\bigr|\,M_j\,\delta
         \;+\;2\,M_j^{2}\,\delta^{2}\Bigr).
\]
\end{lemma}

\begin{proof}
Fix $j$ and put $a=|\fhat(\gamma_j)|$, $b=|\fhat(\tilde\gamma_j)|$.  The
mean value theorem gives $|b-a|\le M_j\delta$, whence
\[
  \bigl|b^{2}-a^{2}\bigr|=|b-a|\,(b+a)
  \le M_j\delta\,\bigl(2a+M_j\delta\bigr)
  =2a\,M_j\delta+M_j^{2}\delta^{2} .
\]
Multiplying by $2$ and summing over $j$ gives the stated bound.
\end{proof}

\begin{remark}[Precision demanded of the ordinates]\label{rem:zeroprec}
Lemma~\ref{lem:zeros} is what fixes the working precision, by the
following sizing.  At a near-minimizer the transform is suppressed at
the low ordinates down to the level of the eigenvalue itself,
$|\fhat(\gamma_j)|\approx\sqrt{\lmin}$, while the derivatives $M_j$ are
of order one on the range that matters; the two terms of the bound are
then of sizes $\delta\sqrt{\lmin}$ and $\delta^{2}$.  Requiring both to
be small against $\lmin$ gives $\delta\lesssim\sqrt{\lmin}$: to resolve
a value of size $\lmin$ the ordinates must be known to accuracy about
$\sqrt{\lmin}$.  This is a sizing estimate rather than a bound with
explicit constants, since $|\fhat(\gamma_j)|\approx\sqrt{\lmin}$ is an
observed feature of the computed minimizers and not something proved
here; it is used only to choose $\delta$, and like the rest of this
subsection it does not enter the unconditional claims
(Remark~\ref{rem:zerostatus}).
\end{remark}

Accordingly, all $1000$ zeros were computed to $115$ significant
digits with \texttt{mpmath} ($\delta^2\approx10^{-220}$, sufficient
for $\lmin\ge10^{-192}$, i.e.\ for $L\le1.8$), and recomputed to $155$
digits for the $L=2.0$ run ($\delta^{2}\approx10^{-300}$); working
precisions were $60$--$330$ digits depending on $L$
(Table~\ref{tab:main}).

\begin{remark}[Status of the zero data]\label{rem:zerostatus}
The zeros-side assembly of this subsection is used only to
\emph{propose} trial vectors and to cross-check the geometric
certificates of Section~\ref{sec:geocert}.  The zero values themselves
are high-confidence multiprecision computations
(\texttt{mpmath.zetazero}), not formally verified enclosures; a fully
rigorous zeros-side certificate would additionally verify each
enclosure (interval Newton or the argument principle) and certify
completeness below $\Gamma_0$ (Turing's method).  Neither step is out
of reach: the verification of Platt and Trudgian
\cite{PlattTrudgian2021} places every zero of height at most
$3\cdot10^{12}$ on the critical line and, by Turing's method, fixes
the zero count below any such height, so both the reality of our
$1000$ zeros and the completeness of the list below
$\Gamma_0=1419.8$ are facts in the literature rather than hypotheses;
what remains is only to propagate that verification into certified
enclosures at our working precision, and the appeal to RH proper is
confined to the tail $\gamma>\Gamma_0$.  None of this enters
the unconditional claims of Table~\ref{tab:main}.
\end{remark}

\begin{definition}[Tail model]\label{lem:tail}
For $r>\Gamma_0>\alpha_{\max}$ expand
$\fhat(r)=-2\cos(rL)\sum_{j\ge0}M_j\,r^{-2j-2}$ with moments
$M_j=\sum_k a_k\alpha_k^{2j+1}$.  The assembly replaces the tail
$\sum_{\gamma>\Gamma_0}2|\fhat(\gamma)|^{2}$ by the model value
\[
  T_{\rm model}
  \;=\;
  \Bigl(\sum_{j,j'\le J}M_jM_{j'}\,C_{j+j'}\Bigr)(1+\varepsilon_J),
  \qquad
  C_m=\frac{4}{\pi}\,
  \frac{\Gamma_0^{-(2m+3)}}{2m+3}
  \Bigl(\ln\frac{\Gamma_0}{2\pi}+\frac{1}{2m+3}\Bigr),
\]
obtained by bounding $\cos^2\le1$, replacing the zeros beyond
$\Gamma_0$ by the smooth density $\nu(t)\,dt$, and truncating the
moment expansion at depth $J$ (with geometric remainder
$\varepsilon_J\le(\alpha_{\max}/\Gamma_0)^{2J+2}/(1-(\alpha_{\max}/\Gamma_0)^2)$).
\end{definition}

Definition~\ref{lem:tail} is a model, not a lemma: replacing the
zero-counting measure $dN$ by the smooth density $\nu\,dt$ is not an
inequality, because $N(t)$ fluctuates around its smooth part.  The
model is used only to assemble the matrix whose bottom eigenvector
serves as a trial vector, and no certified claim depends on it.  In the
certification step (Section~\ref{sec:ivclosure}) the tail is bounded
by a genuine inequality that carries the Backlund bound on
$N(t)-\bar N(t)$ explicitly.  The choice of the depth $J$ is delicate,
not for the certificate, which is valid for any $J$, but for the
quality of the trial vector; see Remark~\ref{rem:caught}.  We use $J$
from $12$ at small $L$ up to $80$ at $L=2.0$, validated by doubling
$J$ and checking stability of the certified value.

\begin{remark}[Proposal versus certificate]\label{rem:cert}
The minimal eigenvalue of $A/L$ is a proposal, contaminated by the tail
model and by solver rounding.  The certified upper bound for $\lmin(L)$
is produced only by the interval re-evaluation of its eigenvector in
Section~\ref{sec:ivclosure}, which uses exact transforms below
$\Gamma_0$ and the rigorous tail inequality above it.  By
Remark~\ref{rem:variational} that re-evaluation is valid for any
coefficient vector whatsoever, so no property of the proposal needs to
be trusted.
\end{remark}

\subsection{Interval-arithmetic closure}\label{sec:ivclosure}

The assembly and the eigensolve run in plain (non-interval)
multiprecision; their rounding errors are dozens of guard digits below
the target scale, but ``far below'' is not ``certified''.  We
therefore close the loop as follows.  The eigenvector $a$ returned by
the solver is treated as a mere proposal, since by
Remark~\ref{rem:variational} any coefficient vector whatsoever yields a
valid upper bound, and the Rayleigh quotient
$R(a)=Q(f_a)/(L\|a\|^2)$ is re-evaluated from scratch in interval
arithmetic (\texttt{mpmath.iv}):

\begin{enumerate}
\item each zero enters as an enclosing interval
  $[\gamma_j-\delta_j,\gamma_j+\delta_j]$ covering both the accuracy
  of the stored digits and the parsing round-off;
\item the finite sum over the first $1000$ zeros uses only the
  closed-form transforms \eqref{eq:shat}, so every operation is an
  interval operation on exact inputs;
\item the tail $\sum_{\gamma>\Gamma_0}2|\fhat(\gamma)|^2$ is bounded
  through the signed series $S(t)=\sum_{j\ge0}M_j t^{-2j-2}$, so that
  $2\fhat(t)^{2}\le8S(t)^{2}$ with the trigonometric factor bounded by
  one.  Writing $B(t)=1+\ln t$, which dominates
  $|N(t)-\frac{t}{2\pi}\ln\frac{t}{2\pi e}-\frac78|$ on the range used
  (Backlund; see \cite[\S9.4]{Titchmarsh1986}), integration by parts
  against $dN$ gives
  \[
    \sum_{\gamma>\Gamma_0}8S(\gamma)^{2}
    \;\le\;8\!\int_{\Gamma_0}^{\infty}\!\!S^{2}\nu\,dt
    \;+\;8S(\Gamma_0)^{2}B(\Gamma_0)
    \;+\;16\Bigl(\int_{\Gamma_0}^{\infty}\!\!S^{2}B\Bigr)^{\!1/2}
    \Bigl(\int_{\Gamma_0}^{\infty}\!\!S'^{2}B\Bigr)^{\!1/2},
  \]
  the last factor by Cauchy--Schwarz applied to $\int16|SS'|B$.  Each
  integral is a finite signed double sum over the moments with
  closed-form kernels.  Keeping the signs matters: the cross-$j$
  cancellation in these sums runs dozens of orders of magnitude deep,
  so any bound routed through $\sum_j|M_j|$ would be useless.  The
  truncation of the series at $J$ carries a geometric remainder, sized
  (adaptively in $J$) below $10^{-12}$ of the target value.
\end{enumerate}

The upper endpoint of the resulting interval is a rigorous variational
upper bound for $\lmin(L)$, conditional on RH and on the correctness of
the zero values.  We retain these zeros-side enclosures as
cross-checks; the unconditional certificates of Table~\ref{tab:main}
are the geometric-side enclosures of the next subsection.

\begin{remark}[The closure caught a real error]\label{rem:caught}
The interval re-evaluation is not a formality; at $L\ge1.8$ it
corrected the record in both directions.  With the moment tail of
Definition~\ref{lem:tail} truncated at $J\approx10$--$14$, which is
ample for $L\le1.6$, the $L=2.0$ eigenvector drifted into directions
whose tail penalty the truncated model underestimates, the relative
slack being of order $z^{J+1}$ with
$z=(\alpha_{\max}/\Gamma_0)^{2}\approx0.37$ at $L=2.0$.  The certified
re-evaluation exposed the resulting proposal $8.2\times10^{-281}$ as
unattainable: the true value of that Rayleigh quotient is
$2.5\times10^{-276}$, four orders of magnitude larger.  Conversely, at
$L=1.8$ the shallow model hid genuinely better directions, and
deepening to $J=40$ improved the certified bound from
$5.2\times10^{-184}$ to $9.0\times10^{-186}$; at $L=2.0$, deepening to
$J=80$ settled it at $4.2\times10^{-283}$, and doubling $J$ again
reproduces every displayed digit.  The mechanism behind all of this is
that the sought eigenvalue lies hundreds of orders of magnitude below
the matrix norm, so a model error even at relative level $10^{-11}$
completely rewires the bottom of the spectrum.  This is the
multiprecision analogue of the float64 pathologies of
Section~\ref{sec:failures}: an eigensolver will exploit any
unpenalized direction of its model, and only a bound re-derived
independently of the model deserves trust.
\end{remark}

\begin{remark}[Scope of the certification]\label{rem:scope}
What is certified is precisely this: each displayed number is an upper
bound for the infinite-dimensional infimum $\lmin(L)$.  What is not
certified is the reverse inequality, the proximity of the bound to
$\lmin(L)$.  The evidence for proximity is structural
rather than rigorous: the sine system is a complete orthogonal basis
of $L^2(-L,L)$, so the Galerkin values decrease monotonically to
$\lmin(L)$ as $K\to\infty$, and the observed stabilization under
dimension refinement (Table~\ref{tab:checks}) indicates that the limit
has been reached at the displayed precision; moreover, at $L=0.8$ the
unconditional lower bound of Theorem~\ref{thm:L08} confirms proximity
within a factor $2.6$ (Section~\ref{sec:synthesis}).  The empirical
law of Section~\ref{sec:results} is thus established for the certified
bounds; its identification with the asymptotics of $\lmin$ itself is
part of Conjecture~\ref{conj:law}.
\end{remark}

\subsection{Geometric-side certification}\label{sec:geocert}

The unconditional certificate evaluates the Rayleigh quotient of the
same trial vector $f_a=\sum_k a_k s_k$ entirely on the geometric side
of the explicit formula.  Writing $g=f_a\star\tilde f_a$ and
$N_2=\|f_a\|_2^2=L\sum a_k^2$,
\begin{equation}\label{eq:Qgeo}
\begin{aligned}
Q(f_a)
&=2\Bigl(\int_{-L}^{L}f_a\cosh\tfrac x2\,dx\Bigr)^{\!2}
-\bigl(\gamma_E+\log\pi+\log(1-e^{-4L})\bigr)N_2\\
&\quad+\int_0^{2L}\frac{2\bigl(e^{-2x}N_2-e^{-x/2}g(x)\bigr)}{1-e^{-2x}}\,dx
-\sum_{\log n<2L}\frac{2\Lambda(n)}{\sqrt n}\,g(\log n).
\end{aligned}
\end{equation}
Every term is a finite sum or an integral over a compact interval; no
zeros and no RH enter.  The autocorrelation of the odd-$k$ cosine
modes admits a closed form, so the prime sum is exact (modulo
interval rounding).  The archimedean integral is split into a
closed-form piece and a remainder that is analytic at the origin;
the remainder is integrated by dyadic-panel Gauss--Legendre with
rigorous Bernstein-ellipse remainders.  The upper endpoint of the
resulting interval enclosure of $Q(f_a)/N_2$ is the unconditional
bound reported in Table~\ref{tab:main}.

At every $L\in[0.5,2.0]$ the geometric enclosure agrees with the
zeros-side proposal to relative accuracy far below the displayed
digits, and is typically a few percent sharper (the zeros-side
pipeline carries a positive tail majorant that the geometric
evaluation does not need).  Relative enclosure widths range from
$10^{-61}$ at $L=0.5$ to $10^{-37}$ at $L=2.0$; the cancellation among
the $O(1)$ blocks of \eqref{eq:Qgeo}, which at $L=2$ reaches some $280$
digits, is fully resolved by the working precision of
Section~\ref{sec:pipeline}.

\subsection{Consistency checks}

\begin{table}[t]
  \centering
  \small
  \begin{tabular}{lll}
    \toprule
    check & result & note\\
    \midrule
    float64 bound at $L=0.7$ &
    $5.0\times10^{-13}$ vs $5.5\times10^{-13}$ &
    independent method, $10\%$\\
    $L=1.2$ rerun, $115$-digit zeros &
    $110.49$ vs $110.49$ & zero-precision insensitivity\\
    parity at $L=1.4$ &
    odd $176.61$ vs even $167.04$ & minimizer is even $\fhat$\\
    $K\!\to\!K+32$ at $L=1.6$ &
    $276.38$ vs $276.45$ & $0.03\%$\\
    $K\!\to\!K+24$ at $L=0.9$ &
    $51.18$ vs $51.32$ & $0.3\%$\\
    interval closure, all $L$ &
    Table~\ref{tab:main} values & Remark~\ref{rem:caught}\\
    lower bound at $L=0.8$ &
    $8.9\times10^{-18}\le2.27\times10^{-17}$ & Thm.~\ref{thm:L08}\\
    \bottomrule
  \end{tabular}
  \caption{Internal consistency checks of the upper-bound pipeline.}
  \label{tab:checks}
\end{table}

The checks are collected in Table~\ref{tab:checks}.  The cost of the
entire computation is modest: the largest run ($L=2.0$: $549$ odd
modes, $330$-digit arithmetic, $1000$ zeros at $155$ digits,
eigenvectors and interval closure included) takes under two hours on
one desktop core with the \texttt{gmpy2} backend.

\section{Results: the Landau--Widom law}\label{sec:results}

\begin{table}[t]
  \centering
  \begin{tabular}{rrrrrrr}
    \toprule
    $L$ & $n$ & dps & $\lmin(L)\le$ & $-\ln\lmin$ & $N(\Tstar)$ &
    $\dfrac{-\ln\lmin\cdot\ln N}{N}$\\
    \midrule
    0.5 & 24 & 90 & $1.05\times10^{-6}$ & 13.76 & 1 & ---\\
    0.6 & 24 & 90 & $1.82\times10^{-9}$ & 20.13 & 1 & ---\\
    0.7 & 24 & 90 & $4.99\times10^{-13}$ & 28.33 & 3 & 10.37\\
    0.8 & 24 & 90 & $2.27\times10^{-17}$ & 38.32 & 4 & 13.28\\
    0.9 & 31 & 95 & $5.91\times10^{-23}$ & 51.18 & 6 & 15.28\\
    1.0 & 40 & 110 & $8.08\times10^{-30}$ & 66.99 & 8 & 17.42\\
    1.1 & 53 & 110 & $2.78\times10^{-38}$ & 86.48 & 12 & 17.90\\
    1.2 & 70 & 125 & $9.98\times10^{-49}$ & 110.53 & 16 & 19.16\\
    1.4 & 119 & 155 & $1.91\times10^{-77}$ & 176.65 & 30 & 20.02\\
    1.6 & 200 & 260 & $8.61\times10^{-121}$ & 276.46 & 55 & 20.15\\
    1.8 & 333 & 420 & $7.86\times10^{-186}$ & 426.22 & 96 & 20.27\\
    2.0 & 549 & 560 & $3.19\times10^{-283}$ & 650.47 & 165 & 20.13\\
    \bottomrule
  \end{tabular}
  \caption{Unconditional certified upper bounds for $\lmin(L)$, from
    geometric-side interval evaluation of $Q(f_L)/\|f_L\|_2^2$
    (Section~\ref{sec:geocert}).  Here $n$ is the number of odd sine
    modes retained, dps is the geometric working precision, $N(\Tstar)$
    is the exact zero count below $\Tstar=2\pi e^{2L}$, and the last
    column is the normalized per-zero cost, which plateaus at
    $20.1\approx2\pi^2=19.74$.  The constant dimension $n=24$ for
    $L\le0.8$ is the floor of $48$ in the rule of
    Section~\ref{sec:pipeline}.}
  \label{tab:main}
\end{table}

\begin{figure}[t]
  \centering
  \includegraphics[width=\textwidth]{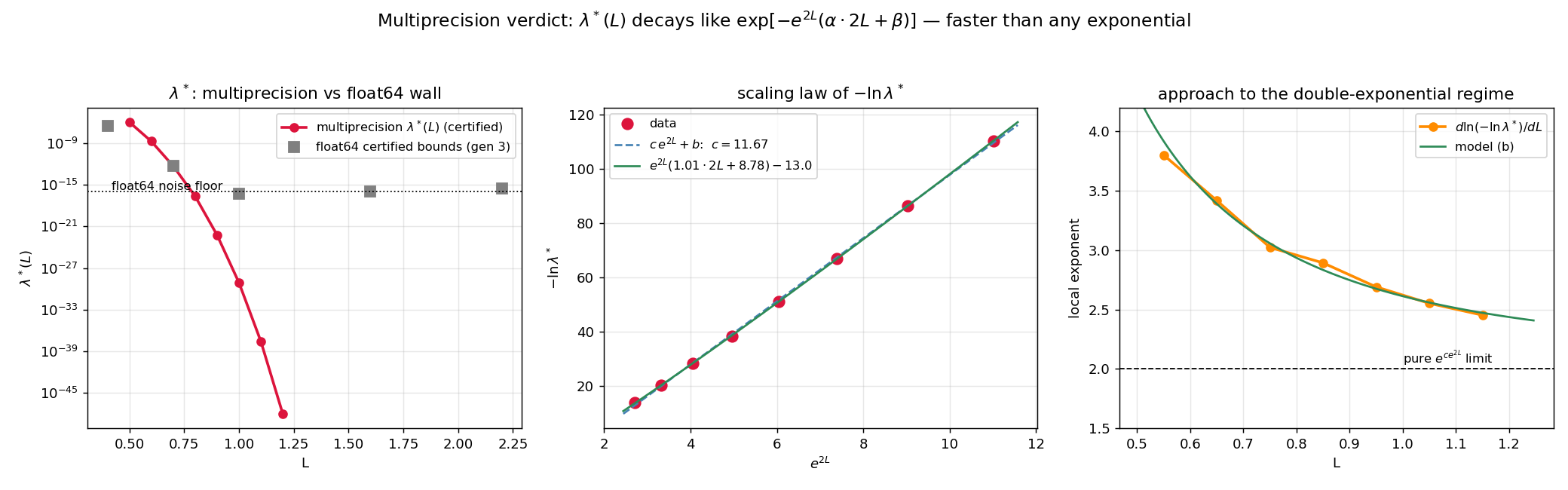}
  \caption{Multiprecision results for $L\le1.2$: the certified bounds
    (left) fall through the float64 wall; $-\ln\lmin$ against $e^{2L}$
    (centre); the local exponent $d\ln(-\ln\lmin)/dL$ decreasing
    towards its asymptote (right).}
  \label{fig:verdict}
\end{figure}

\begin{figure}[t]
  \centering
  \includegraphics[width=\textwidth]{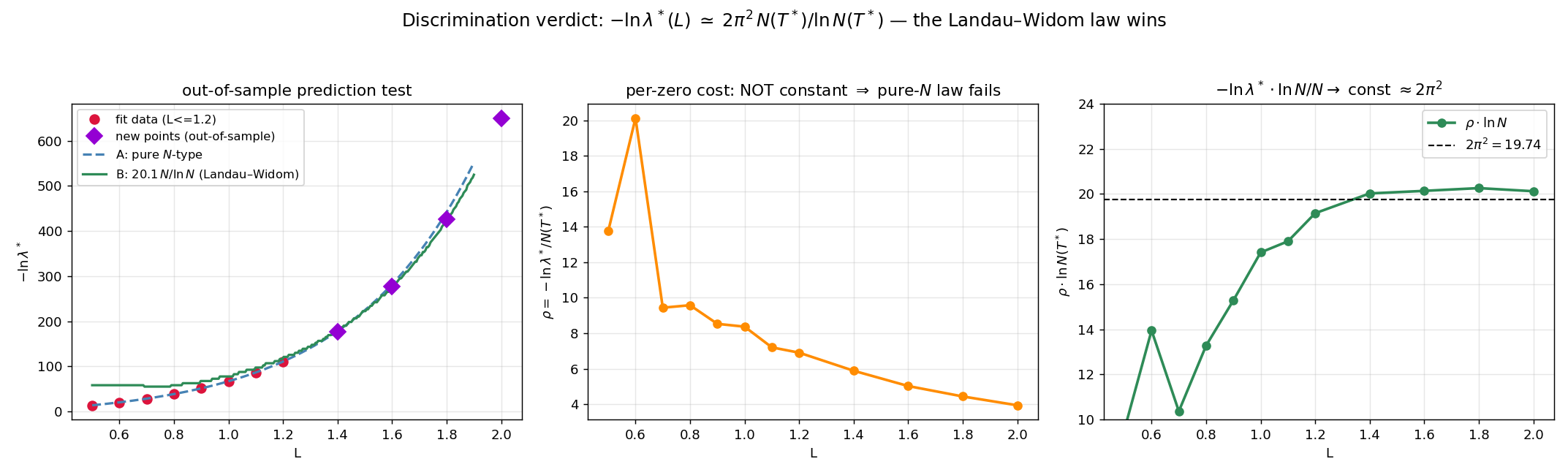}
  \caption{Model discrimination.  Left: out-of-sample prediction of
    the new points $L=1.4,1.6,1.8$ by the two laws calibrated on
    $L\le1.2$.  Centre: the per-zero cost $-\ln\lmin/N(\Tstar)$ is not
    constant, excluding the pure-$N$ law.  Right: multiplied by
    $\ln N(\Tstar)$ it plateaus at $20.1\approx2\pi^{2}$.}
  \label{fig:discrimination}
\end{figure}

\begin{figure}[t]
  \centering
  \includegraphics[width=\textwidth]{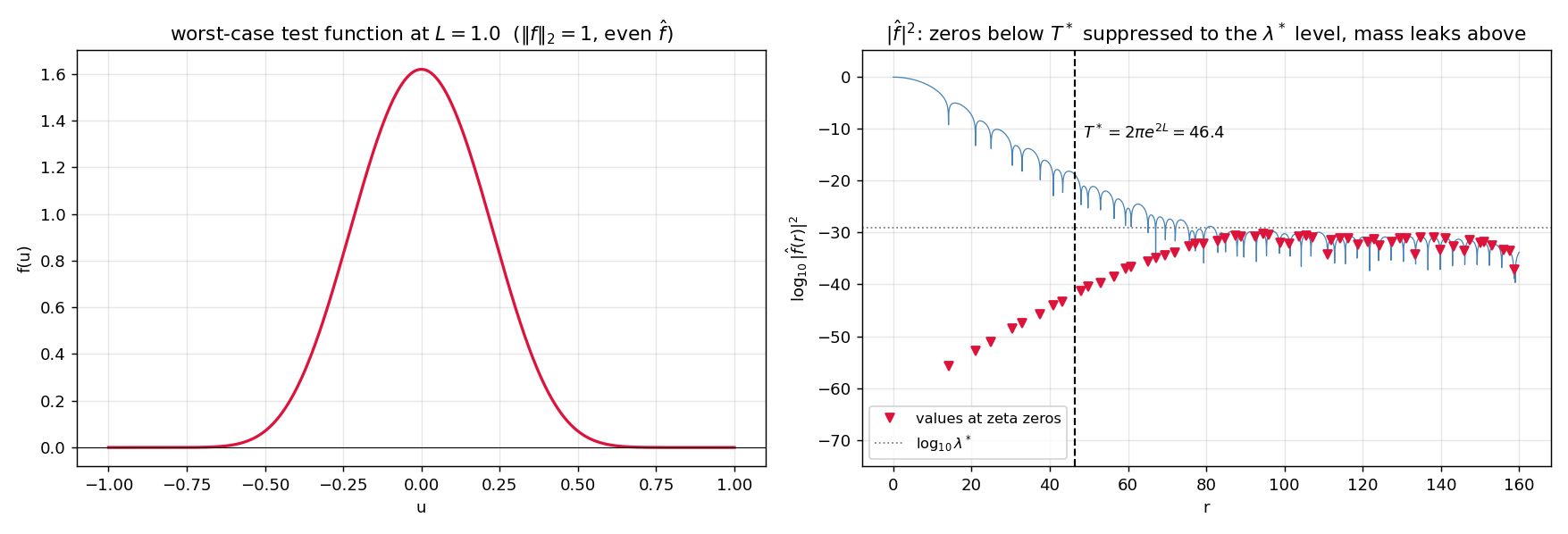}
  \caption{Portrait of the minimizer at $L=1.0$.  Left: the worst-case
    test function.  Right: $|\fhat|^2$ on a logarithmic scale; the
    values at zeta zeros below $\Tstar$ are suppressed to the level
    $\lmin$, and the cost is paid in the transition region around
    $\Tstar$.}
  \label{fig:portrait}
\end{figure}

Table~\ref{tab:main} lists the certified bounds;
Figures~\ref{fig:verdict}--\ref{fig:portrait} display them.  In this
section ``$\lmin$'' refers to these certified upper bounds; per
Remark~\ref{rem:scope}, every statement below is established for the
bounds and conjectural for the infimum itself.  Three observations.

\emph{(i) The law.}  The per-zero cost $-\ln\lmin/N(\Tstar)$ decreases
steadily (from $9.4$ at $L=0.7$ to $3.9$ at $L=2.0$), excluding a law
proportional to $N(\Tstar)$; multiplied by $\ln N(\Tstar)$ it
plateaus: $20.02,\,20.14,\,20.26,\,20.12$ at the last four points,
about $2\%$ above $2\pi^2=19.74$.  A single-constant fit to the four
asymptotic points gives
\[
  -\ln\lmin(L)\;=\;(20.13\pm0.10)\,\frac{N(\Tstar)}{\ln N(\Tstar)},
  \qquad \frac{20.13}{2\pi^{2}}=1.020,
\]
with in-sample residuals below $0.7\%$; the parameter-free value
$C=2\pi^{2}$ reproduces the same four points to within
$1.4\%$--$2.6\%$.  These are fit residuals, and must be distinguished
from the out-of-sample test of Section~\ref{sec:intro-upper}.  Because
the plateau constant is still drifting over the calibration range
($19.15$ at $L=1.2$ in the last column, against $20.02$--$20.26$
beyond), any single constant calibrated on $L\le1.2$ necessarily
underpredicts the new points, here by about $5\%$.  What discriminates
the two laws is that these residuals carry no trend, whereas those of
the pure-$N$ model grow monotonically.  We read the residual offset
from $2\pi^2$, between $1.4\%$ and $2.6\%$ on the plateau and with no
visible trend at these sizes, as an $O(\ln\ln N/\ln N)$ finite-size
correction of the Landau--Widom asymptotics.

\emph{(ii) The rejected model.}  The three-parameter model
$-\ln\lmin=e^{2L}(\alpha\cdot2L+\beta)+\delta_0$ fits the range
$L\le1.2$ with maximal residual $0.2$, and its leading coefficient
$\alpha=1.01$ seemed to identify a ``one nat per zero'' law.  The
out-of-sample test rejects it: its prediction error grows monotonically
($+0.8\%$, $+2.0\%$, $+3.7\%$, $+5.8\%$), whereas the one-parameter
Landau--Widom form, calibrated on the same range, errs by $-4.4\%$,
$-4.9\%$, $-5.5\%$, $-4.8\%$, larger in size but with no trend and with
two fewer parameters.  We record the comparison because the naive
reading goes the other way: a multi-parameter model fitting a short
range to high accuracy is weaker evidence than a one-parameter law
with an identifiable constant that predicts new data.

\emph{(iii) The minimizer.}  Figure~\ref{fig:portrait} shows the
minimizing test function at $L=1.0$: $|\fhat|^{2}$ dips to the working
level at every zero below $\Tstar\approx46$, and releases its mass
just above $\Tstar$, in quantitative agreement with the interpolation
picture of Section~\ref{sec:mechanism}.

\begin{conjecture}\label{conj:law}
As $L\to\infty$,
\[
  -\ln\lmin(L)\;=\;2\pi^{2}\,\frac{N(\Tstar)}{\ln N(\Tstar)}
  \,\bigl(1+o(1)\bigr),
  \qquad \Tstar=2\pi e^{2L},\quad
  N(\Tstar)=e^{2L}(2L-1)+O(L).
\]
\end{conjecture}

\begin{remark}[the constant is fitted]
\label{rem:rival}
The shape $N/\ln N$ is what the data above discriminate; the constant
$2\pi^{2}$ is read off the plateau and is not derived here.  Writing
$R_{1}=-\ln\lmin/\bigl(2\pi^{2}N/\ln N\bigr)$, the numbers of
Table~\ref{tab:main} give
\[
  R_{1}=1.015,\,1.020,\,1.027,\,1.020
  \qquad (L=1.4,1.6,1.8,2.0),
\]
flat to within $1.3\%$, which is the evidence for
Conjecture~\ref{conj:law} at these sizes.  But $R_{1}$ has overshot $1$
and is drifting slowly upwards, and the residual offset from $1$, which
observation (i) reads as an $O(\ln\ln N/\ln N)$ finite-size correction,
is also what a law with a slightly larger asymptotic constant, or a
different slowly varying factor, would produce.  Four points at $L\le2$
cannot separate these readings; we state Conjecture~\ref{conj:law} as
the reading of the present data.
\end{remark}

\section{A conditional double-exponential upper bound}
\label{sec:theorem}

We now prove Theorem~\ref{thm:main-intro}.  Its scope is limited: the
theorem captures only the qualitative content of the empirical law,
namely decay faster than exponential, at least double-exponential with
exponent $e^{L}$.  Neither the scale $\Tstar=2\pi e^{2L}$ nor the
constant $2\pi^{2}$ follows from it; both remain conjectural.
Throughout, RH is assumed, so that \eqref{eq:zerosside} holds, and we
write $T=2\pi e^{L}$ and $\nu(t)=\frac1{2\pi}\ln\frac{t}{2\pi}$, so
that $\nu(T)=\frac{L}{2\pi}$, half of the type budget $L/\pi$.  The
construction thus operates at half saturation, and that is what limits
the exponent to $e^{L}$.

\subsection{Standard facts}

We use two unconditional facts and one classical constant.  Write
\begin{equation}\label{eq:RvM}
  N(t)\;=\;\frac{t}{2\pi}\ln\frac{t}{2\pi e}\;+\;\frac78\;+\;R(t),
  \qquad |R(t)|\;\le\;C_0\ln(t+2)\quad(t\ge0),
\end{equation}
the Riemann--von Mangoldt formula with Backlund's bound on the
remainder \cite[\S9.3--9.4]{Titchmarsh1986}.  Note that
$\frac{d}{dt}\bigl(\frac{t}{2\pi}\ln\frac{t}{2\pi e}\bigr)=\nu(t)$.
Second, from $\sum_{\rho}\rho^{-1}(1-\rho)^{-1}=2+\gamma_E-\ln4\pi$
\cite[Ch.~12]{Davenport2000} and $\rho(1-\rho)=\gamma^2+\tfrac14$ on
RH,
\begin{equation}\label{eq:gammasum}
  \sum_{\gamma>0}\frac{1}{\gamma^{2}}
  \;\le\;\Bigl(1+\frac{1}{4\gamma_1^{2}}\Bigr)
  \sum_{\gamma>0}\frac{1}{\gamma^{2}+\frac14}
  \;=\;\Bigl(1+\frac{1}{4\gamma_1^{2}}\Bigr)
  \frac{2+\gamma_E-\ln4\pi}{2}\;<\;0.0232 .
\end{equation}
Third, since the $\le C_1L$ zeros lying in $[T,T+1]$ cut that interval
into at most $C_1L+1$ gaps, one of the gaps has length
$\ge1/(C_1L+1)$; replacing $T$ by the midpoint of that gap moves $T$
by at most $1$ (changing $N(T)$, $\nu(T)$ and $LT$ by $O(L)$, which is
absorbed in every error term below) and ensures
\begin{equation}\label{eq:gap}
  \mathrm{dist}\bigl(T,\{\gamma_j\}\bigr)\;\ge\;\frac{c_0}{L}
\end{equation}
for an absolute $c_0>0$.  We assume \eqref{eq:gap} from now on.

\subsection{Construction}

Set
\[
  \fhat_0(r)\;=\;B(r)\,h(r),\qquad
  B(r)=\!\!\prod_{0<\gamma_j\le T}\!\Bigl(1-\frac{r^2}{\gamma_j^2}\Bigr),
  \qquad
  h(r)=\Bigl(\frac{\sin(Lr/m)}{Lr/m}\Bigr)^{m},\quad
  m=\Bigl\lceil \frac{LT}{e}\Bigr\rceil .
\]
$B$ is an even polynomial of degree $2N(T)$ (exponential type $0$)
vanishing at every zero up to $T$; $h$ is even, entire of exponential
type exactly $m\cdot(L/m)=L$, with
$|h(r)|\le\min\{1,(m/L|r|)^{m}\}$ on $\R$.  By \eqref{eq:RvM},
$N(T)=e^{L}(L-1)+O(L)$, while $m\ge LT/e=(2\pi/e)Le^{L}>2.31\,Le^{L}$,
so $m>2N(T)+2$ for all large $L$; hence
$\fhat_0(r)=O(|r|^{2N(T)-m})=O(|r|^{-2})$ and
$\fhat_0\in L^{1}\cap L^{2}(\R)$.  By the Paley--Wiener theorem,
$f_0=\mathcal F^{-1}\fhat_0$ is real, even, continuous and supported
in $[-L,L]$; the polynomial decay of $\fhat_0$ places $f_0$ in the
admissible class for the explicit formula (see
\cite[\S2]{Bombieri2000}), so that, under RH,
\[
  Q(f_0)\;=\;\sum_{\gamma>0}2\,\fhat_0(\gamma)^{2}
  \;=\;\sum_{\gamma>T}2\,\fhat_0(\gamma)^{2},
\]
the zeros up to $T$ being annihilated by $B$.

\subsection{The potential bound for the product}

\begin{lemma}[Potential bound]\label{lem:potential}
Let $x\ge1+c_0/(LT)$.  Then
\[
  \ln B(xT)\;\le\;\nu(T)\,T\,\bar g(x)\;+\;C\,L\,(L+\ln 2x),
  \qquad
  \bar g(x)=\int_{0}^{1}\Bigl(\ln\frac{x^{2}-u^{2}}{u^{2}}\Bigr)_{\!+}du,
\]
with an absolute constant $C$.  The profile $\bar g$ has the closed
forms
\[
  \bar g(x)=x\,\ln(3+2\sqrt2)\quad(1\le x\le\sqrt2),
  \qquad
  \bar g(x)=(x{+}1)\ln(x{+}1)-(x{-}1)\ln(x{-}1)\quad(x\ge\sqrt2),
\]
and satisfies $\bar g(x)\le 2+2\ln x$ for $x\ge\sqrt2$.
\end{lemma}

\begin{proof}
Write $r=xT$ and $\varphi(t)=\ln\frac{r^{2}-t^{2}}{t^{2}}$ for
$0<t\le T<r$, so that $\ln B(r)=\int_{0}^{T}\varphi(t)\,dN(t)$ as a
Stieltjes integral.  Let
$\widetilde R(t)=N(t)-\frac{t}{2\pi}\ln\frac{t}{2\pi e}$; by
\eqref{eq:RvM}, $|\widetilde R(t)|\le C_0\ln(t+2)+1$, and
$\widetilde R(t)=-\frac{t}{2\pi}\ln\frac{t}{2\pi e}\to0$ as $t\to0^+$.
Then
\begin{equation}\label{eq:split}
  \ln B(r)\;=\;\int_{0}^{T}\varphi(t)\,\nu(t)\,dt
  \;+\;\int_{0}^{T}\varphi(t)\,d\widetilde R(t).
\end{equation}

\emph{Smooth part.}  We claim
$\varphi(t)\nu(t)\le\varphi(t)_{+}\,\nu(T)$ pointwise on $(0,T]$.
Indeed, $\nu$ is increasing with $\nu(t)\le\nu(T)$; if
$\varphi(t)\ge0$ the claim follows at once (also when $\nu(t)<0$,
i.e.\ $t<2\pi$); and $\varphi(t)<0$ forces $t>r/\sqrt2\ge T/\sqrt2
>2\pi$, where $\nu(t)>0$, making the left side negative.  Hence,
substituting $t=uT$ and using $r=xT$,
\[
  \int_{0}^{T}\varphi\,\nu\,dt
  \;\le\;\nu(T)\int_{0}^{T}\varphi(t)_{+}\,dt
  \;=\;\nu(T)\,T\int_{0}^{1}
  \Bigl(\ln\frac{x^{2}-u^{2}}{u^{2}}\Bigr)_{\!+}du
  \;=\;\nu(T)\,T\,\bar g(x).
\]

\emph{Remainder.}  Integrating by parts (the boundary term at $0$
vanishes since $\widetilde R(t)=O(t\ln\frac1t)$ while
$\varphi(t)=O(\ln\frac1t)$ there),
\[
  \int_{0}^{T}\varphi\,d\widetilde R
  \;=\;\varphi(T)\,\widetilde R(T)
  \;-\;\int_{0}^{T}\varphi'(t)\,\widetilde R(t)\,dt .
\]
For the boundary term, $\varphi(T)=\ln(x^{2}-1)$; by the hypothesis
$x-1\ge c_0/(LT)$ we get
$|\ln(x^2-1)|\le\ln(LT/c_0)+\ln2x+C=O(L+\ln 2x)$, and
$|\widetilde R(T)|=O(\ln T)=O(L)$.  For the integral,
$\varphi'(t)=-\frac{2t}{r^{2}-t^{2}}-\frac{2}{t}$ has constant sign,
and
\[
  \int_{0}^{T}\frac{2t}{r^{2}-t^{2}}\,dt
  =\ln\frac{x^{2}}{x^{2}-1}
  = O(L+\ln 2x),
  \qquad
  \int_{\gamma_1}^{T}\frac{2}{t}\,dt=O(L),
\]
while on $(0,\gamma_1)$ one has
$\frac2t|\widetilde R(t)|\le\frac1\pi(\ln\frac{2\pi}{t}+1)$, which is
integrable with $O(1)$ integral.  Multiplying by
$\sup_{t\le T}|\widetilde R(t)|=O(L)$ where appropriate, the remainder
is $O\bigl(L(L+\ln 2x)\bigr)$, proving the bound.

\emph{Closed forms.}  The integrand of $\bar g$ is positive exactly
for $u<x/\sqrt2$.  If $x\ge\sqrt2$ this is all of $[0,1]$ and a direct
computation gives
$\int_0^1\ln(x^2-u^2)\,du=(x{+}1)\ln(x{+}1)-(x{-}1)\ln(x{-}1)-2$ and
$-2\int_0^1\ln u\,du=2$.  If $1\le x\le\sqrt2$, the substitution
$u=xv$ gives
$\bar g(x)=x\int_{0}^{1/\sqrt2}\bigl(\ln(1-v^{2})-2\ln v\bigr)dv
=x\,\bar g(1)$, and evaluating the antiderivative
$v\ln(1-v^{2})+\ln\frac{1+v}{1-v}-2v\ln v$ at $v=1/\sqrt2$ (the two
$\ln2$ terms cancel) yields
$\bar g(1)=\ln\frac{\sqrt2+1}{\sqrt2-1}=\ln(3+2\sqrt2)$; the two
closed forms agree at $x=\sqrt2$ because
$\ln(\sqrt2+1)=-\ln(\sqrt2-1)=\tfrac12\ln(3+2\sqrt2)$.  Finally,
$\psi(x):=2+2\ln x-\bar g(x)$ satisfies
$\psi'(x)=\frac2x-\ln\frac{x+1}{x-1}=\frac2x-2\,\mathrm{artanh}\frac1x
<0$ and $\psi(x)\to0$ as $x\to\infty$, so $\psi>0$ on
$[\sqrt2,\infty)$.
\end{proof}

\subsection{Mass of the candidate}

\begin{lemma}[Mass]\label{lem:mass}
For all large $L$,\; $\int_{0}^{1}\fhat_0(r)^{2}\,dr\ge0.9$, and hence
$\|f_0\|_{2}^{2}=\frac1{2\pi}\int_{\R}\fhat_0^{2}\ge\frac{0.9}{\pi}
\ge\frac14$.
\end{lemma}

\begin{proof}
Let $0\le r\le1$.  Since $\gamma_1>14$, each factor of $B$ satisfies
$0<1-r^{2}/\gamma_j^{2}<1$, and by $\ln(1-y)\ge-\frac{y}{1-y}$ for
$0\le y<1$, together with \eqref{eq:gammasum},
\[
  \ln B(r)\;\ge\;-\frac{r^{2}}{1-\gamma_1^{-2}}
  \sum_{\gamma>0}\frac{1}{\gamma^{2}}
  \;\ge\;-\,0.0234,
\]
so $B(r)^{2}\ge e^{-0.047}\ge0.953$.  For the window, put
$z=Lr/m\le L/m\to0$; from
$\frac{\sin z}{z}\ge1-\frac{z^{2}}{6}\ge e^{-z^{2}/4}$ (valid for
$0\le z\le1.6$) we get $h(r)\ge e^{-mz^{2}/4}\ge e^{-L^{2}/(4m)}\to1$.
Hence $\int_0^1\fhat_0^2\ge0.953\,(1-o(1))\ge0.9$ for large $L$, and
the $L^{2}$ statement follows from Plancherel and evenness.
\end{proof}

\subsection{Pointwise decay and the profile \texorpdfstring{$F$}{F}}

\begin{lemma}[Pointwise tail bound]\label{lem:pointwise}
For every zero $\gamma=xT>T$,
\[
  \ln|\fhat_0(\gamma)|\;\le\;LT\,F(x)+C\,L\,(L+\ln 2x),
  \qquad
  F(x)=\frac{\bar g(x)}{2\pi}-\frac{1}{e}\,\ln(ex).
\]
\end{lemma}

\begin{proof}
By \eqref{eq:gap}, $x-1=(\gamma-T)/T\ge c_0/(LT)$, so
Lemma~\ref{lem:potential} applies and gives, with $\nu(T)T=LT/2\pi$,
$\ln B(xT)\le\frac{LT}{2\pi}\bar g(x)+O(L(L+\ln2x))$.  For the window,
$LxT/m\ge ex/(1+e/LT)>1$, so $|h(xT)|\le(m/LxT)^{m}$ and
\[
  \ln|h(xT)|\;\le\;-m\ln\frac{LxT}{m}
  \;\le\;-m\Bigl(\ln(ex)-\frac{e}{LT}\Bigr)
  \;\le\;-\frac{LT}{e}\,\ln(ex)+2,
\]
using $LT/e\le m\le LT/e+1$ and $\ln(ex)\ge1$.  Adding the two bounds
proves the lemma.
\end{proof}

\begin{lemma}[The profile is uniformly negative]\label{lem:profile}
$\displaystyle\sup_{x\ge1}F(x)=F(1)=\frac{\ln(3+2\sqrt2)}{2\pi}
-\frac1e=-0.0873\ldots$, and
\[
  F(x)\;\le\;-\Bigl(\frac1e-\frac1\pi\Bigr)(1+\ln x)
  \;\le\;-0.0495\,(1+\ln x)
  \qquad(x\ge\sqrt2).
\]
\end{lemma}

\begin{proof}
\emph{On $[1,\sqrt2]$.}  Here $F(x)=\frac{\kappa}{2\pi}x
-\frac{1+\ln x}{e}$ with $\kappa=\ln(3+2\sqrt2)$, so
$F''(x)=1/(ex^{2})>0$: $F$ is convex and attains its maximum at an
endpoint.  Numerically $F(1)=-0.08733\ldots$ and
$F(\sqrt2)=-0.09862\ldots$, so the maximum is $F(1)$.

\emph{On $[\sqrt2,\infty)$.}  Now
$F'(x)=\frac1{2\pi}\ln\frac{x+1}{x-1}-\frac1{ex}
=\frac{H(x)-2\pi}{2\pi e x}$ with
$H(x)=2ex\,\mathrm{artanh}\frac1x$.  Since
$\mathrm{artanh}\frac1x=\sum_{k\ge0}\frac{x^{-2k-1}}{2k+1}
<\sum_{k\ge0}x^{-2k-1}=\frac{x}{x^{2}-1}$, we get
$H'(x)=2e\bigl(\mathrm{artanh}\frac1x-\frac{x}{x^{2}-1}\bigr)<0$: $H$
is strictly decreasing, from $H(\sqrt2)=e\sqrt2\,\kappa=6.776>2\pi$ to
$\lim_{x\to\infty}H=2e=5.437<2\pi$.  Hence $F$ increases up to the
unique point $x_1$ with $H(x_1)=2\pi$ and decreases afterwards.  From
$H(1.6)=6.377>2\pi>6.238=H(1.7)$ we get $x_1\in(1.6,1.7)$.  Since
$\bar g$ and $\ln$ are increasing, on any subinterval $[a,b]$ one has
$F\le\bar g(b)/2\pi-(1+\ln a)/e$; applying this on $[1.6,1.625]$,
$[1.625,1.65]$, $[1.65,1.675]$, $[1.675,1.7]$ gives the four bounds
$-0.0908$, $-0.0909$, $-0.0910$, $-0.0911$, so
$F(x_1)\le-0.0908<F(1)$.  Together with the monotonicity on
$[\sqrt2,1.6]$ and $[1.7,\infty)$ this proves $\sup_{x\ge1}F=F(1)$.

\emph{Logarithmic decay.}  For $x\ge\sqrt2$,
Lemma~\ref{lem:potential} gives $\bar g(x)\le2+2\ln x$, whence
$F(x)\le\frac{2+2\ln x}{2\pi}-\frac{1+\ln x}{e}
=-(\frac1e-\frac1\pi)(1+\ln x)$, and
$\frac1e-\frac1\pi=0.04957\ldots>0.0495$.
\end{proof}

\subsection{Proof of the theorem}

\begin{proof}[Proof of Theorem~\ref{thm:main-intro}]
Write $\theta=-2F(1)=0.1746\ldots$, so that $2LTF(1)=-\theta LT$.
Split the zeros $\gamma=x_\gamma T>T$ at $x_\gamma=e$.

\emph{Near zone ($1<x_\gamma\le e$).}  By
Lemmas~\ref{lem:pointwise}--\ref{lem:profile}, each term satisfies
$2\ln|\fhat_0(\gamma)|\le2LTF(1)+O(L^{2})$, and by \eqref{eq:RvM} the
number of such zeros is $N(eT)=O(T\ln T)=e^{O(L)}$.  Hence
\[
  \sum_{T<\gamma\le eT}2\,\fhat_0(\gamma)^{2}
  \;\le\;e^{-\theta LT+O(L^{2})}.
\]

\emph{Far zone ($x_\gamma>e$).}  By Lemma~\ref{lem:profile} and
$1+\ln x\ge2$ there,
\[
  2LTF(x)+CL(L+\ln2x)\;\le\;-0.099\,LT\,(1+\ln x)+CL(L+\ln 2x)
  \;\le\;-\theta LT-\tfrac{1}{50}LT\ln x
\]
for all large $L$, because $0.099(1+\ln x)\ge0.198\ge\theta+0.02$ and
$LT\gg L$ absorbs the error term.  Summing over dyadic blocks
$e2^{k}T<\gamma\le e2^{k+1}T$, which contain $O(2^{k}T(L+k))$ zeros
each,
\[
  \sum_{\gamma>eT}2\,\fhat_0(\gamma)^{2}
  \;\le\;2\,e^{-\theta LT}\sum_{k\ge0}O\bigl(2^{k}T(L+k)\bigr)\,
  e^{-\frac{1}{50}LT(1+k\ln2)}
  \;\le\;e^{-\theta LT+O(L)} .
\]

Combining the two zones, $Q(f_0)\le e^{-\theta LT+O(L^{2})}$.  With
Lemma~\ref{lem:mass},
\[
  \lmin(L)\;\le\;\frac{Q(f_0)}{\|f_0\|_2^2}
  \;\le\;4\,e^{-\theta LT+O(L^2)}
  \;=\;\exp\Bigl(-\bigl(\tfrac{4\pi}{e}-2\ln(3+2\sqrt2)\bigr)
  L\,e^{L}+O(L^{2})\Bigr),
\]
where $\theta\cdot2\pi=\frac{4\pi}{e}-2\ln(3+2\sqrt2)
=1.0974\ldots>1$.  Hence $\lmin(L)\le\exp(-Le^{L})$ for all
$L\ge L_0$.
\end{proof}

\begin{remark}[Sharp form of the bound]\label{rem:sharp}
The proof gives the slightly stronger statement
\[
  -\ln\lmin(L)\;\ge\;
  \Bigl(\frac{4\pi}{e}-2\ln(3+2\sqrt2)\Bigr)L\,e^{L}\,
  \bigl(1+O(L\,e^{-L})\bigr),
  \qquad
  \frac{4\pi}{e}-2\ln(3+2\sqrt2)=1.0974\ldots
\]
The two constants $\ln(3+2\sqrt2)$ (the potential cost of the product
at the edge, Lemma~\ref{lem:potential}) and $1/e$ (the optimal decay
rate of a power-of-sinc window per unit of type) are both forced by
the construction; improving either requires a genuinely different
window, cf.\ Remark~\ref{rem:obstruction}.
\end{remark}

\begin{remark}[The obstruction at full saturation]\label{rem:obstruction}
Running the same computation with $T=2\pi e^{2aL}$ shows that the
construction closes only for $a\lesssim0.65$: at $a=1$ the required
decay rate of the window over $[T,2T]$ exceeds what powers of sinc
achieve, by a factor of about $1.5$.  Reaching the true exponent
$e^{2L}$ requires the optimal trade-off between suppression and decay
for band-limited functions, which is exactly the regime described by
the Landau--Widom asymptotics, and whose $\pi^{2}/\ln$ plunge rate
appears as the constant of Conjecture~\ref{conj:law}.  Turning that
conjecture into a theorem thus appears to be a problem about
concentration operators on unions of many intervals, which we leave
open.  In the reverse direction, a lower bound of the form
$\lmin(L)\ge e^{-CN(\Tstar)}$ valid unconditionally for all $L$ would
imply RH, so the realistic target is the conditional statement.
\end{remark}

\section{The barrier}\label{sec:barrier}

We can now prove Theorem~\ref{thm:barrier} and substantiate
Remarks~\ref{rem:barrier-rate} and~\ref{rem:barrier-interp}.

\subsection{Proof of Theorem~\ref{thm:barrier}}
$\Tsharp>T_1=2\pi e^{A_L}$ is forced in
Theorem~\ref{thm:reduction} because $\beta^*>0$ requires it.  By
Lemma~\ref{lem:sharp} the supremum of the comb equals $A_L$ exactly,
so no pointwise bound on the comb can do better.  By the PNT,
$A_L=\sum_{n<e^{2L}}\frac{2\Lambda(n)}{\sqrt n}\sim4e^L$, so
$T_1=2\pi\exp\bigl((4+o(1))e^L\bigr)$: the certificate size is doubly
exponential in the support.  This part is unconditional.

\subsection{The measured margin (Remark~\ref{rem:barrier-rate})}
The margin that any certificate must resolve is the floor itself, and
its decay is now a measured law rather than folklore: by
Table~\ref{tab:main} and \eqref{eq:law}, the window floor collapses at
the Landau--Widom rate $\exp(-2\pi^2N(\Tstar)/\ln N(\Tstar))$ with
$\Tstar=2\pi e^{2L}$, falling through $17$ orders of magnitude at
$L=0.8$, $48$ at $L=1.2$ and $283$ at $L=2$.  The reduced form $R$
tracks the window floor closely (Section~\ref{sec:synthesis}).  The gap
between $\lammin(R)$ and $\lmin$ is governed by the artificial well set
$W_{\Tsharp}=\{t\le\Tsharp:\Psi_L(t)<\beta^*\}$ created by the
splitting: band-limited functions can concentrate on $W_{\Tsharp}$, and
each Shannon mode $\frac{2L|W_{\Tsharp}|}{2\pi}$ of the well set costs
a definite fraction of a decade of floor.  At $L=0.8$ we measure a
Shannon load $\sim5$ against a gap of one decade, so about $0.2$ decade
per mode; the load is $\sim20$--$25$ in the exploratory $L=1.19$ run of
Section~\ref{sec:L119}.  The logical status here is that the law
\eqref{eq:law} is established for certified upper bounds only.
Converting it into a proved lower bound on the cost of any positivity
certificate would require the sampling theory of low zeta zeros, the
window floor being twice the squared smallest singular value of a
low-zero sampling operator, and that remains open.

\subsection{Interpretation (Remark~\ref{rem:barrier-interp})}
The two exponential scales meet head on.  A certificate of the
one-stroke type must reach the comb-alignment threshold
$T_1=2\pi e^{(4+o(1))e^L}$, because at $t$ slightly beyond $T_1$ the
symbol's lower envelope is barely positive and the certificate must
verify that band-limited functions cannot exploit near-alignments of
$(t\log p)_p$.  Below $T_1$ such alignments do null the symbol; above
$T_1$ they cannot.  But proving a sharper threshold than $T_1$, that
is, proving quantitatively that the comb phases avoid simultaneous
alignment, is a statement about linear forms in $\{\log p\}$, and that
is the same species of information as zero-density estimates encode.
Meanwhile the quantity to be certified collapses at the Landau--Widom
rate in the resolution height $\Tstar=2\pi e^{2L}$.  Certifying
positivity past support $\approx3.2$ ($T_1\approx10^7$) by any
pointwise-envelope method is therefore computationally void.  In this
sense the positivity route, continued past the wall, runs back into the
arithmetic of the zeros it was trying to circumvent.

\section{Failed routes (for the record)}\label{sec:failures2}

Three natural alternatives to Theorem~\ref{thm:reduction} were
implemented and abandoned; we record them because their failure modes
motivate the frequency-space formulation.  (They are the geometric-side
counterparts of the float64 pathologies of
Section~\ref{sec:failures}.)
(1) \emph{Galerkin solve-lemmas} (certify $Q\succeq0$ on a window,
lift by solving $Gx=\varphi$): the inverses are rough in $L^2[-L,L]$,
Galerkin residuals decay only polynomially, and the required
$10^{-80}$ tolerances are unreachable.
(2) \emph{Symbol capping} (replace $\Psi_L$ beyond $T_\delta$ by a
constant minorant): provably fatal, since once the window resolves the
cap the form becomes indefinite (measured $\lammin\approx-2.7$).
Capping the symbol amounts to admitting fake zeros.
(3) \emph{Order-space block certificates} (exact window block + tail
bounds + Schur coupling): the prime terms are shift operators, so any
hard cut in Legendre-order space leaves $O(1)$ coupling across the
interface; the collective edge directions lose $\approx-2.4$ against a
supply that grows only logarithmically.
(4) \emph{Per-prime envelope sharpening}: grouping the comb by primes
sharpens its minimum but cannot lower its maximum, which equals $A_L$
exactly (Lemma~\ref{lem:sharp}).  An earlier draft used the sharpened
constant in the envelope, which is the wrong direction, and the
resulting support-$2.38$ claim is retracted in
Section~\ref{sec:L119}.  The error is an easy one to make, since the
well depths of $\Psi_L$, governed by $\Aeff$, and its high-frequency
floor, governed by $A_L$, are two different extremes of the same comb.
The frequency split, by contrast, has no interface and only
super-exponentially small coupling.  It appears to be the right
formulation, and we suspect, without proving, that it is essentially
unique in this respect.

\section{Discussion}\label{sec:discussion}

\textbf{Interpretation.}  The profile $\lmin(L)$ quantifies the
difficulty of approaching RH through compactly supported test
functions, and this paper determines it from both sides.  An
unconditional certificate pushes the positive-side frontier from
$\log2$ to $1.6$, while certified variational bounds show that the
margin does not merely shrink exponentially but collapses at the
Landau--Widom rate $\exp(-2\pi^{2}N/\ln N)$ in the number $N$ of zeros
resolved by the window.  Any finite-dimensional positivity argument
that operates within a window $[-L,L]$, in particular any truncation in
the Connes--Consani prolate program \cite{ConnesConsani2021}, has to
certify a quantity this small.  Conversely, a spectral construction
whose positivity margin decayed more slowly than \eqref{eq:law} on
windows would contradict our variational upper bounds.

\textbf{The constant $2\pi^2$.}  Its appearance links the arithmetic
variational problem \eqref{eq:lambdastar} to the spectral theory of
time--band limiting in a quantitative way that, to our knowledge, is
new.  It is tempting to read \eqref{eq:law} as saying that near the
resolution height the zeros of $\zeta$ behave, as far as band-limited
suppression is concerned, like a generic saturated sampling set, the
arithmetic entering only through the density $\nu$ and hence through
$N(\Tstar)$, with the constant universal.

\textbf{The two heights.}  The synthesis of the two halves isolates a
structural asymmetry of the Weil form on windows.  Defeating positivity,
that is, finding small Rayleigh quotients, is governed by the resolution
height $\Tstar=2\pi e^{2L}$, singly exponential, because the minimizer
only needs to interpolate the zeros it can resolve.  Certifying
positivity through the geometric side is governed by the alignment
threshold $T_1=2\pi e^{(4+o(1))e^{L}}$, doubly exponential, because the
certificate must exclude worst-case phase alignment of the prime comb,
and by Lemma~\ref{lem:sharp} no pointwise envelope can do better.  The
gap between $e^{2L}$ and $e^{4e^L}$ is the gap between knowing where
the zeros are and proving that the primes cannot conspire.  Closing it
is an arithmetic problem, concerning simultaneous Diophantine
properties of $\{\log p\}$, and not a computational one.

\textbf{Open problems.}  (i) Prove Conjecture~\ref{conj:law},
conditionally on RH, presumably via multi-interval Landau--Widom
asymptotics; the upper-bound half with some constant follows if the
window obstruction of Remark~\ref{rem:obstruction} can be resolved.
(ii) Identify the $O(1/\ln N)$ correction, for which our data at
$L=1.4$--$2.0$ already provide constraints, the last column of
Table~\ref{tab:main} reading $20.02,20.14,20.26,20.12$ there.
(iii) Prove a lower bound on the cost of any window positivity
certificate, not just of pointwise-envelope ones, via
the sampling theory of low zeta zeros.  (iv) Repeat both experiments
for Dirichlet $L$-functions: the law \eqref{eq:law} should hold with
the corresponding zero-counting function if the mechanism is
universal, and the one-stroke reduction applies verbatim with the
comb of the corresponding conductor.  (v) Extend the certificates.  A
valid support-$2.38$ certificate needs $\Tsharp\approx1.2\times10^4$
and $N\approx2\times10^4$ (Section~\ref{sec:L119}), which is compiled
double-double territory, unless the interval route sketched there,
certifying $\min\Psi_L$ directly on the finite range where the
envelope is pessimistic, brings $\Tsharp$ back toward the resolution
scale.  (vi) On the upper-bound side, at $L=2.4$ the predicted bound
is $-\ln\lmin\approx1.5\times10^{3}$ (precision $\sim800$ digits,
dimension $\sim2\times10^{3}$), within reach of a few core-days with
GMP-backed arithmetic.

\section{Reproducibility}\label{sec:repro}

Both pipelines run on a single machine in Python~3 with
\texttt{mpmath} \cite{mpmath} (with the \texttt{gmpy2} backend for the
largest runs).  Lower-bound certification is $\approx40$ CPU-hours in
total, plus about two CPU-hours for the parity-sector runs of
Section~\ref{sec:parity}.  On the upper-bound side, the geometric
certificate at $L\le1.4$ finishes in under two minutes per point; the
$L=2.0$ run ($365$-digit geometric arithmetic, $549$ odd modes) is the
bottleneck.

For the certified claims the accompanying files are part of the proof
rather than a courtesy.  Source for both pipelines, software versions,
zero data, certified matrices and Cholesky factors, run logs, error
budgets and SHA-256 hashes are provided as supplementary material,
together with an independent verifier (\texttt{verify\_certificate.py},
depending on \texttt{mpmath} alone) that re-factorises an archived
matrix, recomputes the Cholesky residual in doubled precision, and
re-derives the quadrature budget $c_{\rm err}$ and the tail constants
$\varepsilon_D,\varepsilon_B$ of \eqref{eq:twoblock}.  On the archived
$M_{200}$ it returns $Q(f)\ge9.0\times10^{-18}\|f\|_2^2$
(the $8.9\times10^{-18}$ of Theorem~\ref{thm:L08} rounds the budgets
down).


\begin{thebibliography}{99}

\bibitem{Weil1952}
A.~Weil,
\emph{Sur les ``formules explicites'' de la th\'eorie des nombres
premiers},
Comm.\ S\'em.\ Math.\ Univ.\ Lund, Tome Suppl\'ementaire (1952),
252--265.

\bibitem{Barner1981}
K.~Barner,
\emph{On A. Weil's explicit formula},
J.~reine angew.\ Math.\ \textbf{323} (1981), 139--152.

\bibitem{Yoshida1992}
H.~Yoshida,
\emph{On Hermitian forms attached to zeta functions},
in: Zeta Functions in Geometry, Adv.\ Stud.\ Pure Math.\ \textbf{21},
1992, 281--325.

\bibitem{Bombieri2000}
E.~Bombieri,
\emph{Remarks on Weil's quadratic functional in the theory of prime
numbers, I},
Atti Accad.\ Naz.\ Lincei Cl.\ Sci.\ Fis.\ Mat.\ Natur.\ Rend.\ Lincei
(9) Mat.\ Appl.\ \textbf{11} (2000), 183--233.

\bibitem{Connes1999}
A.~Connes,
\emph{Trace formula in noncommutative geometry and the zeros of the
Riemann zeta function},
Selecta Math.\ (N.S.) \textbf{5} (1999), 29--106.

\bibitem{ConnesConsani2021}
A.~Connes and C.~Consani,
\emph{Weil positivity and trace formula, the archimedean place},
Selecta Math.\ (N.S.) \textbf{27} (2021), no.~4, Paper No.~77.

\bibitem{ConnesMoscovici2022}
A.~Connes and H.~Moscovici,
\emph{The UV prolate spectrum matches the zeros of zeta},
Proc.\ Natl.\ Acad.\ Sci.\ USA \textbf{119} (2022), no.~22,
e2123174119.

\bibitem{ConnesVanSuijlekom2025}
A.~Connes and W.~D.~van Suijlekom,
\emph{Quadratic forms, real zeros and echoes of the spectral action},
preprint, arXiv:2511.23257 (2025).

\bibitem{CCMZeta2025}
A.~Connes, C.~Consani and H.~Moscovici,
\emph{Zeta spectral triples},
preprint, arXiv:2511.22755 (2025).

\bibitem{Suzuki2026}
M.~Suzuki,
\emph{Weil's quadratic form via the screw function},
preprint, arXiv:2606.09096 (2026).

\bibitem{Groskin2026}
A.~Groskin,
\emph{A finite Guinand--Weil dictionary and archimedean tail order for
the truncated Weil quadratic form},
preprint, arXiv:2607.02828 (2026).

\bibitem{SlepianPollak1961}
D.~Slepian and H.~O.~Pollak,
\emph{Prolate spheroidal wave functions, Fourier analysis and
uncertainty, I},
Bell System Tech.\ J.\ \textbf{40} (1961), 43--63.

\bibitem{LandauWidom1980}
H.~J.~Landau and H.~Widom,
\emph{Eigenvalue distribution of time and frequency limiting},
J.\ Math.\ Anal.\ Appl.\ \textbf{77} (1980), 469--481.

\bibitem{RodgersTao2020}
B.~Rodgers and T.~Tao,
\emph{The de Bruijn--Newman constant is non-negative},
Forum Math.\ Pi \textbf{8} (2020), e6.

\bibitem{Polymath15}
D.~H.~J. Polymath,
\emph{Effective approximation of heat flow evolution of the Riemann
$\xi$ function, and a new upper bound for the de Bruijn--Newman
constant},
Res.\ Math.\ Sci.\ \textbf{6} (2019), no.~31.

\bibitem{CCM2013}
E.~Carneiro, V.~Chandee and M.~B.~Milinovich,
\emph{Bounding $S(t)$ and $S_1(t)$ on the Riemann Hypothesis},
Math.\ Ann.\ \textbf{356} (2013), 939--968.

\bibitem{IK2004}
H.~Iwaniec and E.~Kowalski,
\emph{Analytic Number Theory},
AMS Colloq.\ Publ.\ \textbf{53}, 2004.

\bibitem{Rump2010}
S.~M.~Rump,
\emph{Verification methods: rigorous results using floating-point
arithmetic},
Acta Numerica \textbf{19} (2010), 287--449.

\bibitem{Titchmarsh1986}
E.~C.~Titchmarsh,
\emph{The Theory of the Riemann Zeta-function},
2nd ed., revised by D.~R.~Heath-Brown, Oxford University Press, 1986.

\bibitem{Seip2004}
K.~Seip,
\emph{Interpolation and Sampling in Spaces of Analytic Functions},
University Lecture Series \textbf{33}, American Mathematical Society,
2004.

\bibitem{Davenport2000}
H.~Davenport,
\emph{Multiplicative Number Theory},
3rd ed., revised by H.~L.~Montgomery, Graduate Texts in Mathematics
\textbf{74}, Springer, 2000.

\bibitem{PlattTrudgian2021}
D.~Platt and T.~Trudgian,
\emph{The Riemann hypothesis is true up to $3\cdot10^{12}$},
Bull.\ Lond.\ Math.\ Soc.\ \textbf{53} (2021), 792--797.

\bibitem{mpmath}
F.~Johansson et al.,
\emph{mpmath: a Python library for arbitrary-precision floating-point
arithmetic}, \url{https://mpmath.org}.

\end{thebibliography}
\end{document}